\documentclass[11pt]{article}
\usepackage{amsmath,amssymb,amsthm}
\usepackage{color}
\usepackage[colorlinks=true]{hyperref}
\usepackage{pdfsync}
\usepackage[english]{babel}
\usepackage[utf8]{inputenc}

\usepackage{fancybox}
\usepackage{listings}
\usepackage{mathtools}
\usepackage[bottom]{footmisc}
\usepackage{verbatim}

\usepackage{float}
\usepackage{makeidx}

\normalbaselines

\newtheorem{Theorem}{Theorem}[section]

\newtheorem{Definition}[Theorem]{Definition}

\newtheorem{Proposition}[Theorem]{Proposition}

\newtheorem{Lemma}[Theorem]{Lemma}

\newtheorem{Corollary}[Theorem]{Corollary}

\newtheorem{Remark}[Theorem]{Remark}

\newcommand{\N}{\mathbb N}

\newcommand{\RR}{{{\rm I} \kern -.15em {\rm R} }}

\newcommand{\C}{{{\rm l} \kern -.42em {\rm C} }}

\newcommand{\nat}{{{\rm I} \kern -.15em {\rm N} }}

\newcommand{\be}{\begin{equation}}
\newcommand{\ee}{\end{equation}}
\newcommand{\beq}{\begin{eqnarray}}
\newcommand{\eeq}{\end{eqnarray}}
\newcommand{\beqs}{\begin{eqnarray*}}
\newcommand{\eeqs}{\end{eqnarray*}}
\newcommand{\bt}{\begin{Theorem}}
\newcommand{\et}{\end{Theorem}}
\newcommand{\br}{\begin{Remark}}
\newcommand{\er}{\end{Remark}}
\newcommand{\bc}{\begin{Corollary}}
\newcommand{\ec}{\end{Corollary}}
\newcommand{\bl}{\begin{Lemma}}
\newcommand{\el}{\end{Lemma}}
\newcommand{\bd}{\begin{definition}}
\newcommand{\ed}{\end{definition}}
\newcommand{\bp}{\begin{Proposition}}
\newcommand{\eP}{\end{Proposition}}

\DeclareMathOperator\supp{supp}

\newcommand{\R}{\mathbb R}

\title{Time-delayed opinion dynamics with leader-follower interactions: consensus, stability, and mean-field limits}

\bigskip

\author{Young-Pil Choi \footnote{Department of Mathematics, Yonsei University, 50 Yonsei-Ro, Seodaemun-Gu, Seoul 03722, Republic of Korea (\texttt{ypchoi@yonsei.ac.kr}).} \and 
Chiara Cicolani \footnote{Dipartimento di Ingegneria e Scienze dell'Informazione e Matematica, Universit\`{a} dell'Aquila, Via Vetoio, Loc. Coppito, 67100 L'Aquila, Italy (\texttt{chiara.cicolani@graduate.univaq.it}).} \and  
Cristina Pignotti \footnote{Dipartimento di Ingegneria e Scienze dell'Informazione e Matematica, Universit\`{a} dell'Aquila, Via Vetoio, Loc. Coppito, 67100 L'Aquila, Italy (\texttt{cristina.pignotti@univaq.it}).}}

\begin{document}
\maketitle

\begin{abstract}
We study a time-delayed variant of the Hegselmann--Krause opinion formation model featuring a small group of leaders and a large group of non-leaders. In this model, leaders influence all agents but only interact among themselves. At the same time, non-leaders update their opinions via interactions with their peers and the leaders, with time delays accounting for communication and decision-making lags. We prove the exponential convergence to consensus of the particle system, without imposing smallness assumptions on the delay parameters. Furthermore, we analyze the mean-field limit in two regimes: (i) with a fixed number of leaders and an infinite number of non-leaders, and (ii) with both populations tending to infinity, obtaining existence, uniqueness, and exponential decay estimates for the corresponding macroscopic models.
\end{abstract}

\vspace{5 mm}

\def\qed{\hbox{\hskip 6pt\vrule width6pt
height7pt
depth1pt  \hskip1pt}\bigskip}

\tableofcontents

%
%
%

%
%
%
%
%
%
\section{Introduction}\label{sec1}
In the recent years, the analysis of multi-agents systems became a very attractive topic, because of the large amount of possible applications. They naturally appear in various scientific fields, e.g. in biology \cite{Cama, Carrillo, CS1}, ecology \cite{Sole}, economics \cite{Marsan}, social sciences \cite{Bellomo, BN, BHT, Castellano, Campi, Lac, Dolfin}, physics \cite{DH, Stro}, control theory \cite{Almeida2, D, Pao-Pig, PRT, WCB}, engineering and robotics \cite{Bullo, Desai}. For other applications, see also \cite{Hel, Jack, XWX}.
In the context of multi-agent systems, opinion formation models have been extensively studied to understand how individuals in a society converge to a common viewpoint or polarization occurs. One prominent example is the Hegselmann--Krause (HK) model \cite{HK}, which describes how agents adjust their opinions by averaging their neighbors' opinions within a certain confidence range. The HK model is especially appealing due to its simplicity and its ability to capture a rich variety of collective behaviors observed in social dynamics. Several generalizations of this model have been proposed (see e.g. \cite{Bellomo, BN, CFT, CFT2, Ceragioli, JM}).

One of the most natural extensions of this model concerns the analysis of the time-delayed interactions. Indeed, in many realistic settings, interactions between individuals are not instantaneous. The process of communication, deliberation, and decision-making unavoidably introduces a time delay. Such delays may stem from the finite time needed to receive and process information, or from asynchronous updates among agents. Incorporating these time delays into opinion dynamics models is crucial for better aligning the mathematical description with observed real-world behaviors. In particular, delayed models offer insight into how past interactions shape current decisions and can lead to phenomena such as delayed consensus, sustained oscillations, or transient clustering. Of course, the presence of a delay in the interaction makes the problem more difficult to analyze, since the delay, even if small, can destroy some geometric features of the system. 

Opinion formation models in the presence of time delay effects have already been studied by several authors, and convergence to consensus results have been obtained, first in the case of small time delays  (see e.g.  \cite{CPP, Pao-Pig}) and more recently without any restrictions on the time delay size (see e.g. \cite{3, CPM2A, 4, onoff, Cic-Pig, Cic_Con_Pign}). The second-order version of the model, namely the celebrated Cucker-Smale model, with time delays has been analyzed, e.g., in \cite{CDH22, CH, CL, CP, Cartabia, Cont, DHK20}.

For well-posedness results for alignment models in the presence of time delay effects, we refer to classical texts on functional differential equations \cite{Halanay, Hale}. Here, we
will focus on the asymptotic behavior of the solutions.

In the current work, we investigate a time-delayed opinion formation model in which a population is divided into two groups: a small number of leaders and a larger group of followers (non-leaders). Specifically, we consider $m \in \mathbb{N}$ leaders and $N \in \mathbb{N}$ non-leaders, with
\[
N > m \geq 2.
\]
The defining feature of our setting is the asymmetric interaction structure: while leaders influence all agents in the system, they themselves are influenced only by their fellow leaders. This models situations where opinion leaders, such as political figures, experts, or influencers, shape the views of the broader population but are insulated from the direct influence of non-leaders.

Let
\[
y_i(t) \in \R^d, \quad i=1,\dots, m,
\]
denote the opinion of the $i$-th leader at time $t$, and let
\[
x_i(t) \in \R^d, \quad i=1,\dots, N,
\]
denote the opinion of the $i$-th non-leader at time $t$. To account for the time required for discussion and decision-making, we introduce time delays into the interactions among agents. Consequently, the evolution of opinions is governed by the Hegselmann-Krause opinion formation model with delays:
\begin{equation}\label{eqA}
\begin{array}{l}
\displaystyle{
 \frac{d}{d t}y_i(t)=  \frac{1}{m} \sum_{j=1}^m \psi_{ij}^{\tau_1}(t)(y_j(t-\tau_1)-y_i(t)), \quad t>0, \ i=1,...,m,} \\
\displaystyle{ \frac{d}{d t}x_i(t)=  \frac{1}{N} \sum_{j=1}^N \phi_{ij}^{\tau_2}(t)(x_j(t-\tau_2)-x_i(t))}\\
\displaystyle{\hspace{3 cm} +  \frac{1}{m} \sum_{j=1}^{m} \rho_{ij}^{\tau_1}(t)(y_j(t-\tau_1)-x_i(t)), \quad t>0,\ i=1,...,N,} 
\end{array}
\end{equation}
where the interaction weights are defined by  
\begin{equation} \label{weight}
\begin{split}
& \psi_{ij}^{\tau_1}(t):=\psi(y_i(t),y_j(t-\tau_1)), \\
& \phi_{ij}^{\tau_2}(t):=\phi(x_i(t),x_j(t-\tau_2)), \\
& \rho_{ij}^{\tau_1}(t):=\rho(x_i(t),y_j(t-\tau_1)),
\end{split}
\end{equation}
and 
$
\psi, \phi, \rho: \mathbb{R}^d \times \mathbb{R}^d \rightarrow \mathbb{R}
$
are assumed to be positive, Lipschitz continuous, and uniformly bounded.  We want to emphasize that, in order to prove the convergence to consensus, we do not need the Lipschitz property for the communication rates $\psi, \ \phi, \ \rho.$ We will require the Lipschitz continuity only later on, to deal with the mean-field approximations.  We denote the maximal delay by
\[
\tau:=\max \{\tau_1,\tau_2\},
\]
and prescribe the initial data on the interval $[-\tau,0]$:
\begin{equation} \label{ic1A}
y_i(t)=y_i^0(t), \quad i=1,\dots,m, \ t \in [-\tau,0],
\end{equation}
and 
\begin{equation} \label{ic2A}
x_i(t)=x_i^0(t), \quad i=1,\dots,N, \ t \in [-\tau,0],
\end{equation}
where $y_i^0$ and $x_i^0$ are continuous functions from $[-\tau,0]$ to $\R^d$.

Under the above framework, we aim to show that the system achieves consensus in the long run (cf. \cite{Cic_Oua_Pign} for a related approach with different normalization factors).

To precisely formulate the consensus result, we introduce the following notions.

\begin{Definition} \label{def'}
The global diameter of the system is defined as
\begin{equation} \label{diameter}
\begin{split}
d(t):= \max \Big\{ \max_{i,j=1,\dots,N}&|x_i(t)-x_j(t)|,  \max_{i,j=1,\dots,m}|y_i(t)-y_j(t)|,  \max_{\substack{i=1,\dots,m \\ j=1,\dots,N}} |y_i(t)-x_j(t)| \Big\}.
\end{split}
\end{equation}
Moreover, for each $n \in \mathbb{N}_0$, we define the diameter over the time interval $[n\tau - \tau, n\tau]$ by
\[
\begin{split}
D_n:= \max_{s,t \in [n\tau-\tau,n\tau]} \Big\{ \max_{i,j=1,\dots,N}&|x_i(s)-x_j(t)|,  \max_{i,j=1,\dots,m}|y_i(s)-y_j(t)|,  \max_{\substack{i=1,\dots,m \\ j=1,\dots,N}}|y_i(s)-x_j(t)| \Big\}.
\end{split}
\]
\end{Definition}

We are now in a position to state our first main result regarding the exponential convergence to consensus.
\begin{Theorem}\label{decay_Th}
Let $\{y_i(t)\}_{i=1}^m$ and $\{x_j(t)\}_{j=1}^N$ be the global-in-time classical solution to the system \eqref{eqA} with initial conditions \eqref{ic1A}--\eqref{ic2A}. Then,  there exists a constant $\gamma > 0$, independent of $N$ and $m$, such that the global diameter decays exponentially:
\[
d(t) \le e^{-\gamma (t-2\tau)} D_0.
\]
\end{Theorem}

The proof of Theorem \ref{decay_Th} builds upon the analysis of the global opinion diameter functional $d(t)$, which measures the maximal deviation of opinions across all agents. A central challenge lies in controlling the evolution of $d(t)$ under delayed interactions, especially between leaders and followers. To address this, we derive a Gr\"onwall-type inequality for $d(t)$ by decomposing the interaction terms and estimating their influence across time intervals of length $\tau$. Through a stepwise contraction argument and bootstrapping, we establish that $d(t)$ decays exponentially after a transient layer, independently of the delay size. This approach allows us to obtain consensus convergence without imposing smallness assumptions on the delay parameters.

To understand the collective behavior of a large number of interacting agents, we next study the mean-field limit of the delayed particle system introduced in \eqref{eqA}. The goal of this analysis is twofold: first, to derive macroscopic equations that describe the evolution of the system when the number of agents becomes large; and second, to establish convergence results toward consensus at the mean-field level.

We investigate two distinct asymptotic regimes that reflect different population structures.
\medskip

\noindent{\it Case (i): few leaders and many non-leaders.} In this regime, the number of leaders $m$ remains fixed, while the number of non-leaders $N$ tends to infinity. This setup reflects situations in which a small number of influential individuals shape the dynamics of a much larger population. In the mean-field limit, the leaders retain their finite-dimensional dynamics, whereas the non-leaders are described by a probability density $\nu_t$ governed by a continuity equation. The limiting system reads:
\begin{equation}\label{pde1}
\begin{array}{l}
\displaystyle{
 \frac{d}{d t} \bar y_i(t)=  \frac{1}{m} \sum_{j=1}^m \bar\psi_{ij}^{\tau_1}(t)\bigl( \bar y_j(t-\tau_1)-  \bar y_i(t)\bigr), \quad t>0, \quad i=1,\dots,m,} \\[2mm]
\displaystyle{
 \partial_t \nu_t + \nabla \cdot (\nu_t\, v_t^m)= 0, \quad x \in \mathbb{R}^d, \quad t>0,}
\end{array}
\end{equation}
with the initial data
\[
(\bar y_i(s), \nu_s) =: (\bar y_i^0(s), g_s), \quad i=1,\dots, m, \ x \in \R^d \quad \mbox{for } s \in [-\tau,0].
\]
The interaction weight is given by $ \bar \psi_{ij}^{\tau_1}(t):=\psi(\bar y_i(t), \bar y_j(t-\tau_1))$, and the velocity field for the non-leader density is given by
\begin{equation}\label{flux1}
v_t^m(x):= \int_{\mathbb{R}^d} \phi(x,y)(y-x)\,\nu_{t-\tau_2}(d y) +  \frac{1}{m} \sum_{j=1}^{m}  \rho(x,\bar y_j(t-\tau_1)) \bigl( \bar y_j(t-\tau_1)-x\bigr).
\end{equation}

In this formulation, the leaders remain finite-dimensional agents evolving under delayed mutual interactions, while the non-leader population evolves continuously in time and space under the influence of both the leader group and its own internal dynamics. This hybrid description allows for a tractable yet rich model of hierarchical opinion dynamics.
\medskip

\noindent{\it Case (ii): infinite population limit for both leaders and non-leaders.} In this fully macroscopic regime, both groups are described by probability densities. The mean-field limit then yields a pair of continuity equations for $\bar\mu_t$ (leaders) and $\bar\nu_t$ (non-leaders):
\begin{equation}\label{pde2}
\begin{array}{l}
\displaystyle{
 \partial_t \bar \mu_t + \nabla \cdot (\bar \mu_t\, \bar u_t)= 0, \quad x \in \mathbb{R}^d, \quad t>0,} \\[2mm]
\displaystyle{
 \partial_t \bar \nu_t + \nabla \cdot (\bar \nu_t\, \bar v_t)= 0, \quad x \in \mathbb{R}^d, \quad t>0,}
\end{array}
\end{equation}
subject to the initial data
\begin{equation}\label{init2}
(\bar \mu_s, \bar \nu_s) =: (\bar f_s, \bar g_s), \quad x \in \R^d \quad \mbox{for } s \in [-\tau,0].
\end{equation}
Here, the velocity fields are defined as
\begin{align}\label{flux2}
\bar u_t(x)&:= \int_{\mathbb{R}^d} \psi(x,y)(y-x)\,\bar\mu_{t-\tau_1}(d y),  \\ \label{flux3}
\bar v_t(x)&:= \int_{\mathbb{R}^d} \phi(x,y)(y-x)\,\bar\nu_{t-\tau_2}(d y) +  \int_{\mathbb{R}^d} \rho(x,y)(y-x)\,\bar\mu_{t-\tau_1}(d y). 
\end{align}
This fully macroscopic description is particularly useful for analyzing large-scale patterns and stability properties of the system when the number of interacting agents is extremely high.

To study convergence and decay to consensus in these mean-field models, we introduce diameter-like quantities that measure the spread of the distributions.

In Case (i), we define:
\[
\begin{split}
d^{\nu}(t)&:= \max \Bigl\{ \sup_{x,y \in \operatorname{supp}(\nu_t)}|x-y|, \; \max_{i,j=1,\dots,m}|y_i(t)-y_j(t)|, \;  \max_{i=1,\dots,m}\sup_{x \in \operatorname{supp}(\nu_t)}|y_i(t)- x| \Bigr\}.
\end{split}
\]
In Case (ii), the diameter becomes
\[
\begin{split}
d^{\mu, \nu}(t)&:= \max \Bigl\{ \sup_{x,y \in \operatorname{supp}(\nu_t)}|x-y|, \;  \sup_{x,y \in \operatorname{supp}(\mu_t)}|x-y|,  \; \sup_{\substack{x \in \operatorname{supp}(\nu_t),\\ y \in \operatorname{supp}(\mu_t)}}|x-y| \Bigr\}.
\end{split}
\]
To measure initial discrepancies, we define:
\[
\begin{split}
D^\nu_0 &:= \max_{s,t \in [-\tau,0]}  \max \Bigl\{ \sup_{ \substack{x \in \operatorname{supp}(g_s),\\ y \in \operatorname{supp}(g_t)}}|x-y|, \;  \max_{i,j=1,\dots,m}|y_i(s)-y_j(t)|, \;\max_{i=1,\dots,m}\sup_{x \in \operatorname{supp}(g_t)}|y_i(s)- x| \Bigr\}, \\
D^{\mu, \nu}_0 &:= \max_{s,t \in [-\tau,0]}  \max \Bigl\{ \sup_{\substack{x \in \operatorname{supp}(g_s),\\ y \in \operatorname{supp}(g_t)}}|x-y|, \; \sup_{\substack{x \in \operatorname{supp}(f_s),\\ y \in \operatorname{supp}(f_t)}}|x-y|,  \; \sup_{ \substack{x \in \operatorname{supp}(f_s),\\  y \in \operatorname{supp}(g_t)}}|x-y| \Bigr\}.
\end{split}
\]
 
We now recall the standard notions of push-forward and measure-valued solutions to make the above formulations precise.
\begin{Definition}\label{push}
Let $\mu$ be a Borel measure in $\mathbb{R}^d$ and $\mathcal{T}:\mathbb{R}^d \rightarrow \mathbb{R}^d$ be a measurable map. The push-forward of $\mu$ by $\mathcal{T}$ is the measure $\mathcal{T} \# \mu$ defined by
$$ \mathcal{T} \# \mu (B):= \mu (\mathcal{T}^{-1}(B)),$$
for every Borel sets $B \subset \mathbb{R}^d.$
\end{Definition}

\begin{Definition}\label{solution}
Let $T>0$ and let $\mathcal{P}(\mathbb{R}^d)$ denote a set of probability measures in $\mathbb{R}^d$. We say that $\mu_t \in C([0,T);\mathcal{P}(\mathbb{R}^d) )$ is a measure-valued solution to a continuity equation of the form \eqref{pde1} or \eqref{pde2} if for every $\phi \in C_c^{\infty}(\mathbb{R}^d \times [0,T))$, the following weak formulation holds:
\begin{equation}\label{measure}
\int_{0}^T \int_{\mathbb{R}^d} (\partial_{t}\phi + v(x) \cdot \nabla_x \phi) \mu_t(dx)dt + \int_{\mathbb{R}^d} \phi(x,0) \mu_0(dx) = 0,
\end{equation}
where $v$ is the velocity field defined as \eqref{flux1}, \eqref{flux2}, or \eqref{flux3}.
\end{Definition}

\begin{Theorem}\label{1.7}
Assume that the initial data for the particle system \eqref{eqA} satisfy one of the following:
\begin{enumerate} 
    \item[ ] \textbf{Case (i): few leaders and many non-leaders:} The leader initial data are given by 
    \[
    y_i^0 \in C([-\tau_1, 0]), \quad i=1,\dots,m,
    \]
    and the non-leader initial distribution is 
    \[
    g \in C\bigl([-\tau_2,0]; \mathcal{P}_\infty(\mathbb{R}^d)\bigr).
    \]
    Here $\mathcal{P}_\infty(\R^d)$ denotes the space of all probability measures on $\R^d$ with bounded support.
    \item[ ] \textbf{Case (ii): infinite population limit for both leaders and non-leaders:} The initial densities satisfy
    \[
    f \in C\bigl([-\tau_1,0]; \mathcal{P}_\infty(\mathbb{R}^d)\bigr) \quad \text{and} \quad g \in C\bigl([-\tau_2,0]; \mathcal{P}_\infty(\mathbb{R}^d)\bigr).
    \]
\end{enumerate}
Then, for any finite time $T>0$, the corresponding mean-field model admits a unique solution of equations \eqref{pde1} or \eqref{pde2} on the interval $[0,T)$ with the following regularity properties:
\begin{itemize}
    \item In Case (i), the leader trajectories $\bar y_i(t)$ belong to $C^1([0,T))$, and the non-leader distribution $\nu$ belongs to $C([0,T); \mathcal{P}_{\infty}(\mathbb{R}^d))$, with uniformly compact support. Moreover, the solution satisfies the relation:
    \begin{equation}\label{pushf}
\nu_t = X(t;\cdot) \# \nu_0,
\end{equation}
where $X(t;\cdot)$ is the flow map generated by $v_t^m$.
    
    \item In Case (ii), the measure-valued solutions $\bar \mu(t)$ and $\bar \nu(t)$ belong to $C([0,T);  \mathcal{P}_{\infty}(\mathbb{R}^d))$, with uniformly compact support, and the solutions satisfy
\[
\bar \mu_t = X(t;\cdot) \# \bar \mu_0, \quad \bar \nu_t = Z(t;\cdot) \# \bar \nu_0,
\]
where $X$ and $Z$ are the flow maps associated with $\bar u_t$ and $\bar v_t$, respectively.
\end{itemize}
Moreover, denoting by $d(t)$ the global diameter of the solution (i.e., $d^\nu(t)$ in Case (i) and $d^{\mu,\nu}(t)$ in Case (ii)) and by $D_0$ the corresponding initial discrepancy (i.e., $D^\nu_0$ in Case (i) and $D^{\mu,\nu}_0$ in Case (ii)), there exists a constant $\gamma>0$ such that
\[
d(t) \leq e^{-\gamma(t - 2\tau)}\, D_0, \quad \forall\, t\geq 0.
\]
\end{Theorem}

Theorem \ref{1.7} addresses both the well-posedness and large-time behavior of the mean-field systems derived from the interacting particle dynamics. The proof proceeds in two steps. First, we establish the existence and uniqueness of measure-valued solutions by constructing them as push-forwards of initial measures under characteristic flows. The well-posedness is shown in the space of probability measures endowed with the $p$-Wasserstein distance. A key ingredient is a Lipschitz-type stability estimate, derived using optimal transport techniques, which ensures continuous dependence on the initial data. Second, to analyze the large-time behavior, we combine the exponential decay estimate for the particle system (Theorem \ref{decay_Th}) with a quantitative mean-field limit argument. In both regimes, (i) with finitely many leaders and infinitely many followers, and (ii) with both populations tending to infinity, we prove that the macroscopic dynamics inherit the consensus property from their particle counterparts. This two-step approach highlights the robustness of consensus formation under time delays and scaling limits.

The rest of this paper is organized as follows. In Section \ref{sec2}, we provide the particle-level description of the model with time-delayed interactions and establish existence and large-time behavior results under appropriate assumptions.
Section \ref{sec3} is devoted to the mean-field formulation and the rigorous construction of measure-valued solutions to the limiting system, including the proof of Theorem \ref{1.7} on existence and uniqueness. In Section \ref{sec4}, we analyze the stability and asymptotic behavior of the mean-field system. We first derive stability estimates in the Wasserstein distance, and then rigorously establish the consensus behavior by combining stability with particle approximation techniques.
 
%
%
%
%
%
%
\section{Exponential consensus in the time-delayed particle system} \label{sec2}
In this section, we establish exponential convergence to consensus for the delayed leader-follower system defined in \eqref{eqA}. Our analysis proceeds by constructing suitable upper and lower bounds for directional components of the trajectories and showing that the overall diameter contracts over time. We begin with a preliminary lemma that ensures directional components of agent positions remain uniformly bounded over delayed intervals.

For convenience, we set the uniform bound on interaction strengths:
\begin{equation} \label{K}
 K:= \max \{\| \psi \|_{L^\infty}, \| \phi \|_{L^\infty}, \| \rho \|_{L^\infty} \}.
 \end{equation}

\begin{Lemma} \label{b}
Let $\{x_i(t)\}_{i=1}^N$ and $\{y_j(t)\}_{j=1}^m$ be the solution to \eqref{eqA} with the initial conditions given by \eqref{ic1A}--\eqref{ic2A}. For any vector $v\in\mathbb{R}^d$ and time $T\ge 0,$ define the quantities
\[
m_T:= \min \left\{ \min_{j=1,\dots,N} \min_{s \in [T-\tau,T]}  \langle x_j(s), v \rangle, \min_{j=1,\dots,m} \min_{s \in [T-\tau,T]}  \langle y_j(s), v \rangle \right\},
\]
and 
\[
M_T:= \max \left\{ \max_{j=1,\dots,N} \max_{s \in [T-\tau,T]}  \langle x_j(s), v \rangle, \max_{j=1,\dots,m} \max_{s \in [T-\tau,T]}  \langle y_j(s), v \rangle \right\}.
\]
Then, for all $t\ge T-\tau$:
\begin{equation} \label{eq1}
m_T  \leq \langle x_i(t), v \rangle \leq M_T, \quad i=1,\dots,N,
\end{equation}
and 
\begin{equation} \label{eq2}
m_T \leq \langle y_i(t), v \rangle \leq  M_T, \quad i=1, \dots, m.
\end{equation}
\end{Lemma}
\begin{proof}
We prove \eqref{eq1}; the proof of \eqref{eq2} follows analogously. Fix $v\in\mathbb{R}^d$ and $T\geq 0$. By definition, the inequalities in \eqref{eq1} are trivially satisfied for $t\in[T-\tau, T]$. For $t>T$, let $\epsilon>0$ be arbitrary and define the set
$$ \mathcal{S}_{\epsilon}:= \left\{ t>T : \max_{j=1,\dots,N} \langle x_j(s),v \rangle < M_T +\epsilon, \ s \in [T,t) \right\}. $$
Since the trajectories are continuous, $\mathcal{S}_{\epsilon}$ is nonempty. Denote  
\[
S:= \sup \mathcal{S}_{\epsilon}.
\]
We now claim that $S= + \infty$. Assume, by contradiction, that $S<+\infty$. Then, for all $s\in (T,S)$ we have
\[
\max_{j=1,\dots,N} \langle x_j(s),v \rangle < M_T +\epsilon,
\]
and, by continuity,
\[
\lim_{s \rightarrow S^-} \max_{j=1,\dots,N} \langle x_j(s),v \rangle = M_T +\epsilon.
\]
Now, take any $t\in (T, S)$. Using the second equation in \eqref{eqA} and noting that $t-\tau_1,\,t-\tau_2\in (T-\tau,S)$, we differentiate to obtain
\begin{equation*}
\begin{split}
\frac{d}{d t} \langle x_i(t),v \rangle & =  \frac{1}{N} \sum_{j=1}^N \phi_{ij}^{\tau_2}(t)\langle x_j(t-\tau_2)-x_i(t),v \rangle +  \frac{1}{m} \sum_{j=1}^{m} \rho_{ij}^{\tau_1}(t)\langle y_j(t-\tau_1)-x_i(t),v \rangle \\
& \leq \frac{1}{N} \sum_{j=1}^N \phi_{ij}^{\tau_2}(t)\bigl(M_T+\epsilon-\langle x_i(t),v \rangle\bigr)+\frac{1}{m}\sum_{j=1}^m \rho_{ij}^{\tau_1}(t)\bigl(M_T+\epsilon-\langle x_i(t),v\rangle\bigr) \\
& \leq 2K\bigl(M_T+\epsilon-\langle x_i(t),v\rangle\bigr),
\end{split}
\end{equation*}
where we used the bound $\phi_{ij}^{\tau_2}(t),\,\rho_{ij}^{\tau_1}(t)\le K$. Then,
\[
\frac{d}{d t} \langle x_i(t),v \rangle \leq 2K (M_T+\epsilon)-2K\langle x_i(t),v \rangle,
\]
and applying Gr\"onwall's inequality on $(T,t)$ yields
\[
\langle x_i(t),v \rangle \leq e^{-2K(t-T)} \langle x_i(T),v\rangle+( M_T+\epsilon )\left( 1-e^{-2K(t-T)}\right).
\]
Since $\langle x_i(T),v\rangle \le M_T,$ by the definition of $\mathcal{S}_{\epsilon}$, we find
$$ \max_{i=1,\dots,N} \langle x_i(t),v \rangle \leq M_T+\epsilon-\epsilon e^{-2K(S-T)}, \quad \forall t \in (T,S).$$
Taking the limit as $t \rightarrow S^-,$ we get
$$ \lim_{t \rightarrow S^-} \max_{i=1,\dots,N} \langle x_i(t),v \rangle \leq M_T+\epsilon-\epsilon e^{-2K(S-T)} < M_T+ \epsilon, $$
which contradicts the earlier limit. Therefore, we must have $S=+\infty$, and consequently,
$$  \max_{i=1,\dots,N} \langle x_i(t),v \rangle < M_T+\epsilon, \quad \forall \, t \geq T.$$
Since $\epsilon>0$ is arbitrary, it follows that
$$ \max_{i=1,\dots,N} \langle x_i(t),v \rangle \leq M_T, \quad \forall \, t \geq T.$$ In particular, 
\begin{equation}\label{ineq1}
\langle x_i(t),v \rangle \leq M_T, \quad \forall \, t \geq T, \ i=1,\dots,N.
\end{equation} 
Now, we apply \eqref{ineq1} with a vector $-v \in \mathbb{R}^d$ to get 
\begin{equation*}
\begin{split}
 -\langle x_i(t),v \rangle =  \langle x_i(t),-v \rangle & \leq \max_{j=1,\dots,N} \max_{s \in [T-\tau,T]}  \langle x_j(s),-v \rangle \\
& \leq - \min_{j=1,\dots,N} \min_{s \in [T-\tau,T]}  \langle x_j(s),v \rangle = - m_T. 
\end{split}
\end{equation*}
Then, we have the second inequality 
\[
\langle x_i(t),v \rangle \geq m_T, \quad \forall\, t \geq T, \ i=1,\dots,N.
\]
Thus, \eqref{eq1} is established.
\end{proof}

%
%
%
%
%
%
%
%
\subsection{Uniform boundedness and influence positivity}
Building on Lemma \ref{b}, we now establish uniform bounds on the pairwise distances between agents' opinions. This is key to showing that the diameter of the system does not increase over time and will ultimately decay exponentially.

\begin{Lemma} \label{3}
For each $n \in \mathbb{N}_0,$  the following estimates hold:
\begin{equation}\label{3A}
|x_i(s)-x_j(t)| \leq D_n, \quad \forall \, s,t \geq n\tau-\tau, \ i,j=1,...,N,
\end{equation}
\begin{equation}\label{3B}
|y_i(s)-y_j(t)| \leq D_n, \quad \forall \, s,t \geq n\tau-\tau, \ i,j=1,\dots,m,
\end{equation}
and 
\begin{equation}\label{3C}
|y_i(s)-x_j(t)| \leq D_n, \quad \forall \ s,t \geq n\tau-\tau, \ i=1,\dots,m, \ j=1,\dots,N.
\end{equation}
\end{Lemma}

\begin{Remark} \label{2.3}
An immediate consequence of Lemma \ref{3} is that the sequence $\{D_n\}_{n\in\mathbb{N}_0}$ is non-increasing:
$$
D_{n+1} \leq D_n,\quad \mbox{\rm  for all}\  n \in \mathbb{N}_0.$$
Then, in particular,
\[
|x_i(s)-x_j(t)| \leq D_0, \quad \forall \ i,j=1,...,N, 
\]
\[
|y_i(s)-y_j(t)| \leq D_0, \quad \forall \ i,j=1,...,m,
\]
and
\[
|y_i(s)-x_j(t)| \leq D_0, \quad \forall \ i=1,...,m, \ j=1,\dots,N,
\]
for all $s,t \geq -\tau.$ 
\end{Remark}

\begin{proof}[Proof of Lemma \ref{3}]
We present the proof of \eqref{3A}; the proofs of \eqref{3B} and \eqref{3C} are analogous. Fix $s,t \ge -\tau$ and suppose that $|x_i(s)-x_j(t)|>0$ (the case when this difference is zero is trivial). Define the vector
\[
v:=x_i(s)-x_j(t).
\]
By applying Lemma \ref{b} and the Cauchy-Schwarz inequality, we obtain
\begin{equation*}
\begin{split}
& \vert x_i(s)-x_j(t)\vert^2 =\langle x_i(s)-x_j(t),v \rangle \\
&\quad \leq \max \Big\{ \max_{i=1,\dots,m} \max_{s \in [n\tau-\tau,n\tau]}  \langle y_i(s),v \rangle,\max_{i=1,\dots,N} \max_{s \in [n\tau-\tau,n\tau]}  \langle x_i(s),v \rangle \Big\} \\
& \hspace{3 cm}  - \min \Big\{ \min_{j=1,\dots,m} \min_{t \in [n\tau-\tau,n\tau]} \langle y_j(t),v \rangle, \min_{j=1,\dots,N} \min_{t \in [n\tau-\tau,n\tau]} \langle x_j(t),v \rangle \Big\} \\
&\quad \leq \max \Big\{ \max_{i,j=1,\dots,m} \max_{s,t \in [n\tau-\tau,n\tau]} \langle y_i(s)-y_j(t),v \rangle, \max_{i,j=1,\dots,N} \max_{s,t \in [n\tau-\tau,n\tau]} \langle x_i(s)-x_j(t),v \rangle, \\
& \hspace{3 cm}\max_{i=1,\dots,m} \max_{j=1,\dots,N} \max_{s,t \in [n\tau-\tau,n\tau]} \langle y_i(s)-x_j(t),v \rangle \Big\} \\
&\quad \leq \max \Big\{ \max_{i,j=1,\dots,m} \max_{s,t \in [n\tau-\tau,n\tau]} |y_i(s)-y_j(t)|\cdot |v|, \max_{i,j=1,\dots,N} \max_{s,t \in [n\tau-\tau,n\tau]} |x_i(s)-x_j(t)| \cdot |v|, \\
& \hspace{3 cm}\max_{i=1,\dots,m} \max_{j=1,\dots,N} \max_{s,t \in [n\tau-\tau,n\tau]} |y_i(s)-x_j(t)| \cdot |v| \rangle \Big\} \cr
&\quad \leq D_n^2.
\end{split}
\end{equation*}
This completes the proof.
\end{proof}

We next show that the agents' trajectories remain uniformly bounded in time.

\begin{Lemma} \label{4}
For all $i=1,...,N$ and $j=1,\dots,m$, the solutions of \eqref{eqA} satisfy
\[
|x_i(t)| \leq C_0 \quad \mbox{and} \quad |y_i(t)| \leq C_0, \quad \forall \, t \geq 0,
\]
where 
\begin{equation*} 
C_0:= \max_{s \in [-\tau,0]}\left\{ \max_{i=1,...,N}|x_i(s)|,  \;\max_{i=1,...,m} |y_i(s)| \right\}.
\end{equation*}
\end{Lemma}
\begin{proof}
We prove the bound for $x_i(t)$; the corresponding estimate for $y_j(t)$ follows similarly. Fix $i\in\{1,\dots,N\}$ and $t\ge 0$. If $|x_i(t)|=0$, the bound is trivial. Otherwise, set
\[
v:=x_i(t).
\]
Then, using Lemma \ref{b}, we obtain
\begin{equation*}
\begin{split}
 \vert x_i(t) \vert^2 &=  \langle x_i(t),v \rangle  \\
& \leq \max \Big\{ \max_{j=1,\dots,m} \max_{s \in [-\tau,0]}  \langle y_j(s),v \rangle, \max_{j=1,\dots,N} \max_{s \in [-\tau,0]}  \langle x_j(s),v \rangle \Big\}\\
& \leq \max \Big\{ \max_{j=1,\dots,m} \max_{s \in [-\tau,0]}  |y_j(s)| \cdot |v|, \max_{j=1,\dots,N} \max_{s \in [-\tau,0]} |x_j(s)| \cdot |v|\Big\} \cr
&\leq C_0^2.
\end{split}
\end{equation*}
This completes the proof.
\end{proof}

From Lemma \ref{4} we immediately obtain a useful lower bound for the influence functions. This bound is crucial, as it ensures that the interactions among the agents remain uniformly positive throughout the evolution of the system.

\begin{Remark} \label{2.5}
Since the influence functions are continuous, we deduce that
\begin{equation*}
\begin{split}
& \psi(y_i(t),y_j(t-\tau_1)) \geq \psi_0:=\min_{|z_1|,|z_2| \leq C_0} \psi(z_1,z_2) >0, \\
& \phi(x_i(t),x_j(t-\tau_2))\geq \phi_0:=\min_{|z_1|,|z_2| \leq C_0} \psi(z_1,z_2) >0, \\
& \rho(x_i(t),y_j(t-\tau_1))\geq \rho_0:=\min_{|z_1|,|z_2| \leq C_0} \rho(z_1,z_2) >0,
\end{split}
\end{equation*}
for each $t \geq 0.$ This positivity is crucial since it prevents the influence terms from degenerating, thereby ensuring effective information exchange across the network.
\end{Remark}

%
%
%
%
%
%
%
%
\subsection{Directional contraction estimates}

With these lower bounds in hand, we now turn our attention to establishing contraction estimates that describe how the differences between agent states decay over time. These estimates will serve as the cornerstone of the consensus result.

\begin{Lemma} \label{5}
For all unit vector $ v \in \R^d$ and for every $n \in \N_0$, the following inequalities hold:
\begin{equation}
\langle x_i(t)-x_j(t),v \rangle \leq e^{-2K(t-t_0)} \langle x_i(t_0)-x_j(t_0),v \rangle + (1-e^{-2K(t-t_0)})D_n, \quad \forall \,  i,j=1,\dots,N,
\label{5A}
\end{equation}
\begin{equation}
\langle y_i(t)-y_j(t),v \rangle \leq e^{-2K(t-t_0)} \langle y_i(t_0)-y_j(t_0),v \rangle + (1-e^{-2K(t-t_0)})D_n, \quad \forall\, i,j=1,\dots,m,
\label{5B}
\end{equation}
and, $\forall\, i=1, \dots,N,\ j=1,\dots,m,$
\begin{equation}
\langle x_i(t)-y_j(t),v \rangle \leq e^{-2K(t-t_0)} \langle x_i(t_0)-y_j(t_0),v \rangle + (1-e^{-2K(t-t_0)})D_n, \quad 
\label{5C}
\end{equation}
for all $t \geq t_0 \geq n\tau.$ 
\end{Lemma}
\begin{proof} We divide the proof into two steps. In the first step, we obtain contraction estimates for the differences among agents within the same group (both non-leaders and leaders), and in the second step, we treat the mixed case involving a non-leader and a leader. \\

\emph{\textbf{Step 1.}} We first derive the contraction estimate for the non-leader agents. Fix a unit vector $v \in \R^d$ and a given $n \in \N_0$. To quantify the maximal and minimal projection values along $v$ over the time interval $[n\tau-\tau,\, n\tau]$, define
\[
M_n:= \max_{s \in [n\tau-\tau,n\tau]} \left\{ \max_{j=1,\dots,N}  \langle x_j(s), v \rangle, \;\max_{j=1,\dots,m} \langle y_j(s), v \rangle \right\},
\]
and 
\[
m_n:= \min_{s \in [n\tau-\tau,n\tau]} \left\{ \min_{j=1,\dots,N}  \langle x_j(s), v \rangle, \;\min_{j=1,\dots,m} \langle y_j(s), v \rangle \right\}.
\]
It is clear that $M_n-m_n \leq D_n$. Now, fix an index $i \in \{1,\dots,N\}$ and consider $t \geq t_0 \geq n\tau$. \\ 

By the definition of system \eqref{eqA} and applying Lemma \ref{b}, we have
\begin{equation*}
\begin{split}
\frac{d}{d t} \langle x_i(t),v \rangle & = \frac{1}{N} \sum_{j=1}^N \phi_{ij}^{\tau_2}(t) \langle x_j(t-\tau_2)-x_i(t),v \rangle + \frac{1}{m} \sum_{j=1}^{m} \rho_{ij}^{\tau_1}(t) \langle y_j(t-\tau_1)-x_i(t),v \rangle \\
& \leq \frac{1}{N}\sum_{j=1}^N \phi_{ij}^{\tau_2}(t) \bigl(M_n - \langle x_i(t),v \rangle\bigr) + \frac{1}{m} \sum_{j=1}^{m} \rho_{ij}^{\tau_1}(t) \bigl(M_{n}- \langle x_i(t),v \rangle \bigr) \\
& \leq 2K\bigl(M_n - \langle x_i(t),v \rangle\bigr), \\
\end{split}
\end{equation*}
where we used that $\phi_{ij}^{\tau_2}(t)$ and $\rho_{ij}^{\tau_1}(t)$ are bounded by $K$, and that $t-\tau_1,\,t-\tau_2 \geq n\tau-\tau$. By applying the Gr\"onwall's lemma, we find 
\begin{equation}\label{eq21}
\langle x_i(t),v \rangle \leq e^{-2K(t-t_0)}  \langle x_i(t_0),v \rangle + (1-e^{-2K(t-t_0)})M_n .
\end{equation}
Similarly, for any $j \in \{1,\dots,N\}$ and $t \geq t_0 \geq n\tau$, we derive the lower bound
\begin{equation}\label{eq22}
\langle x_j(t),v \rangle \geq e^{-2K(t-t_0)}  \langle x_j(t_0),v \rangle + (1-e^{-2K(t-t_0)})m_n .
\end{equation}
Subtracting \eqref{eq22} from \eqref{eq21} yields
\begin{equation*}
\begin{split}
\langle x_i(t)-x_j(t),v \rangle & \leq e^{-2K(t-t_0)}  \langle x_i(t_0)-x_j(t_0),v \rangle + (1-e^{-2K(t-t_0)})(M_n-m_n) \\
& \leq e^{-2K(t-t_0)}\langle x_i(t_0)-x_j(t_0),v \rangle + (1-e^{-2K(t-t_0)})D_n.
\end{split}
\end{equation*}
Thus, we have obtained the contraction estimate for the non-leader agents as stated in \eqref{5A}. By an analogous argument applied to the leader dynamics (using the first equation of \eqref{eqA} and the corresponding influence function $\psi$), one obtains a similar contraction estimate \eqref{5B} for the leader agents. \\

\emph{\textbf{Step 2.}} We now consider the mixed case involving a non-leader and a leader. For a leader $y_j(t)$, we use the first equation in system \eqref{eqA} and apply Lemma \ref{b} to obtain
\begin{equation*}
\begin{split}
\frac{d}{d t} \langle y_j(t),v \rangle & = \frac{1}{m} \sum_{l=1}^m \psi_{jl}^{\tau_1}(t) \langle y_l(t-\tau_1)-y_j(t),v \rangle \\
& \geq \frac{1}{m}\sum_{l=1}^m \psi_{jl}^{\tau_1}(t) (m_n - \langle y_j(t),v \rangle) \\
& \geq K(m_n - \langle y_j(t),v \rangle) \\
& \geq 2K(m_n - \langle y_j(t),v \rangle),
\end{split}
\end{equation*}
where we used $m_n - \langle y_j(t),v \rangle\le 0$ for all $j=1,\dots,m$ and for all $t\ge n\tau.$ Applying Gr\"onwall's inequality then gives
\begin{equation}\label{eq22A}
\langle y_j(t),v \rangle \geq e^{-2K(t-t_0)}  \langle y_j(t_0),v \rangle + (1-e^{-2K(t-t_0)})m_n.
\end{equation}
Subtracting \eqref{eq22A} from the bound for $\langle x_i(t), v \rangle$ in \eqref{eq21} yields
\begin{equation*}
\begin{split}
\langle x_i(t)-y_j(t),v \rangle & \leq e^{-2K(t-t_0)}  \langle x_i(t_0)-y_j(t_0),v \rangle + (1-e^{-2K(t-t_0)})(M_n-m_n) \\
& \leq e^{-2K(t-t_0)}\langle x_i(t_0)-y_j(t_0),v \rangle + (1-e^{-2K(t-t_0)})D_n.
\end{split}
\end{equation*}
This completes the proof of inequality \eqref{5C}.
\end{proof}

\begin{Remark}
It is worth noting that for pairs of leaders, one may also derive a sharper contraction estimate that depends only on the leader group:
\begin{equation*}
\langle y_i(t)-y_j(t),v \rangle \leq e^{-K(t-t_0)} \langle y_i(t_0)-y_j(t_0),v \rangle + (1-e^{-K(t-t_0)})\max_{h,k=1,\dots, m}\max_{r,s\in [n\tau-\tau, n\tau]} \vert y_h(r)-y_k(s)\vert,
\label{5D}
\end{equation*}
 for all $i,j=1,\dots,m$. However, for the overall consensus result, it is essential to work with unified estimates (as in Lemma \ref{5}) that simultaneously control all interactions in the mixed leader-follower system.
\end{Remark}

%
%
%
%
%
%
%
%
\subsection{Recursive control of diameter}
To estimate the evolution of the system diameter in discrete time, we now investigate how the maximum distance between any two agents at time $n\tau$ relates to earlier diameters. The following lemma provides a key step toward establishing exponential contraction of the global diameter.

\begin{Lemma} \label{6}
There exists a constant $C \in (0,1)$ such that 
\[
d(n\tau) \leq C D_{n-2},
\]
for all $n \geq 2.$
\end{Lemma}
\begin{proof} We prove the lemma by considering three distinct cases, each corresponding to a different configuration in which the diameter $d(n\tau)$ is achieved.

\medskip

\textbf{\emph{Case 1.}} Assume that 
\[
d(n\tau)=|x_i(n\tau)-x_j(n\tau)| 
\]
for some $i,j=1,\dots,N.$ Since the case $|x_i(n\tau)-x_j(n\tau)|=0$ is trivial, we suppose $|x_i(n\tau)-x_j(n\tau)|>0.$ In this case, we first normalize the difference by setting
\[
v:=\frac{x_i(n\tau)-x_j(n\tau)}{|x_i(n\tau)-x_j(n\tau)|}.
\]

To capture the spread of agent opinions over a preceding time interval, we now introduce the quantities
\[
  M_{n-1}:=  \max_{s \in [(n-2)\tau,(n-1)\tau]} \left\{ \max_{j=1,\dots,N}  \langle x_j(s), v \rangle, \;\max_{j=1,\dots,m} \langle y_j(s), v \rangle \right\},
\]
            and 
\[
  m_{n-1}:=  \min_{s \in [(n-2)\tau,(n-1)\tau]} \left\{ \min_{j=1,\dots,N}  \langle x_j(s), v \rangle, \;\min_{j=1,\dots,m} \langle y_j(s), v \rangle \right\}.
\]
It is clear that 
\[
M_{n-1}-m_{n-1} \leq D_{n-1}.
\]

Next, we analyze the evolution of the projection differences along $v$ during the time interval $t \in [(n-1)\tau,n\tau]$. Using the system dynamics, we write the time derivative of the projection difference between agents $x_i$ and $x_j$ as follows:
\begin{align*}
& \frac{d}{d t} \langle x_i(t)-x_j(t), v \rangle  \cr
&\quad = \frac{1}{N} \sum_{l=1}^N \phi_{il}^{\tau_2}(t) \langle x_l(t-\tau_1)-x_i(t), v \rangle + \frac{1}{m}\sum_{l=1}^{m}\rho_{il}^{\tau_1}(t) \langle y_l(t-\tau_2)-x_i(t),v \rangle\\
&\qquad - \frac{1}{N}\sum_{l=1}^N \phi_{jl}^{\tau_2}(t) \langle x_l(t-\tau_1)-x_j(t), v \rangle - \frac{1}{m}\sum_{l=1}^{m} \rho_{jl}^{\tau_1}(t) \langle y_l(t-\tau_2)-x_j(t), v \rangle.
\end{align*}
We now regroup the terms by introducing the shift $M_{n-1}$ (which represents an upper bound on the projections during $[(n-2)\tau,(n-1)\tau]$). In particular, we rewrite the above derivative as
\begin{equation}
\begin{split}
& \frac{d}{d t} \langle x_i(t)-x_j(t), v \rangle  \cr
&\quad = \frac{1}{N}\sum_{l=1}^N \phi_{il}^{\tau_2}(t) (\langle x_l(t-\tau_1), v \rangle - M_{n-1}) + \frac{1}{N}\sum_{l=1}^N \phi_{il}^{\tau_2}(t) (M_{n-1}-\langle x_i(t), v \rangle) \\
&\qquad + \frac{1}{m} \sum_{l=1}^{m} \rho_{il}^{\tau_1}(t) (\langle y_l(t-\tau_2), v \rangle - M_{n-1}) + \frac{1}{m} \sum_{l=1}^{m} \rho_{il}^{\tau_1}(t) (M_{n-1}- \langle x_i(t), v \rangle) \\
&\qquad + \frac{1}{N}\sum_{l=1}^N \phi_{jl}^{\tau_2}(t) (m_{n-1}-\langle x_l(t-\tau_1), v \rangle) + \frac{1}{N}\sum_{l=1}^N \phi_{jl}^{\tau_2}(t) (\langle x_j(t), v \rangle-m_{n-1}) \\
&\qquad  + \frac{1}{m}\sum_{l=1}^m \rho_{jl}^{\tau_1}(t) (m_{n-1}-\langle y_l(t-\tau_2), v \rangle) + \frac{1}{m}\sum_{l=1}^m \rho_{jl}^{\tau_1}(t) (\langle x_j(t), v \rangle-m_{n-1}) .
\end{split}
\label{68} 
\end{equation}
At this point, it is convenient to introduce two auxiliary sums, $S_1$ and $S_2$, corresponding respectively to the contributions involving the upper bound $M_{n-1}$ and the lower bound $m_{n-1}$. Using Remark \ref{2.5} and the fact that the weights are bounded by $K$, we obtain
\begin{equation*}
\begin{split}
S_1 & := \frac{1}{N}\sum_{l=1}^N \phi_{il}^{\tau_2}(t) (\langle x_l(t-\tau_1), v \rangle - M_{n-1}) + \frac{1}{N}\sum_{l=1}^N \phi_{il}^{\tau_2}(t) (M_{n-1}-\langle x_i(t), v \rangle) \\
&\quad + \frac{1}{m} \sum_{l=1}^{m} \rho_{il}^{\tau_1}(t) (\langle y_l(t-\tau_2), v \rangle - M_{n-1}) + \frac{1}{m} \sum_{l=1}^{m} \rho_{il}^{\tau_1}(t) (M_{n-1}- \langle x_i(t), v \rangle) \\
& \leq \frac{\phi_0}{N} \sum_{l=1}^N (\langle x_l(t-\tau_1), v \rangle - M_{n-1})  + \frac{\rho_0}{m} \sum_{l=1}^{m}(\langle y_l(t-\tau_2), v \rangle - M_{n-1}) \\
&\quad + 2K (M_{n-1}-\langle x_i(t), v \rangle).
\end{split}
\end{equation*} 

Similarly, we define
\begin{equation*}
\begin{split}
S_2 & := \frac{1}{N}\sum_{l=1}^N \phi_{jl}^{\tau_2}(t) (m_{n-1}-\langle x_l(t-\tau_1), v \rangle) + \frac{1}{N}\sum_{l=1}^N \phi_{jl}^{\tau_2}(t) (\langle x_j(t), v \rangle-m_{n-1}) \\
& + \frac{1}{m}\sum_{l=1}^m \rho_{jl}^{\tau_1}(t) (m_{n-1}-\langle y_l(t-\tau_2), v \rangle) + \frac{1}{m}\sum_{l=1}^m \rho_{jl}^{\tau_1}(t) (\langle x_j(t), v \rangle-m_{n-1}) \\
& \leq \frac{\phi_0}{N} \sum_{l=1}^N (m_{n-1}-\langle x_l(t-\tau_1), v \rangle)+ \frac{\rho_0}{m}\sum_{l=1}^m  (m_{n-1}-\langle y_l(t-\tau_2), v \rangle) \\
& + 2K(\langle x_j(t), v \rangle - m_{n-1}),
\end{split}
\end{equation*}
where we used the fact that, being $t \in [(n-1)\tau,n\tau],$ it holds that  $t -\tau_1, t-\tau_2 \in [(n-2)\tau,n\tau].$
Combining the estimates from $S_1$ and $S_2$ in \eqref{68}, we deduce that
\begin{equation*}
\begin{split}
  \frac{d}{d t} \langle x_i(t)-x_j(t), v \rangle 
& \leq 2K (M_{n-1}-m_{n-1}) - 2K \langle x_i(t)-x_j(t), v \rangle \\
&\quad + \frac{\phi_0}{N}  \sum_{l=1}^N(\langle x_l(t-\tau_1), v \rangle - M_{n-1})  + \frac{\rho_0}{m} \sum_{l=1}^{m}(\langle y_l(t-\tau_2), v \rangle - M_{n-1}) \\
&\quad + \frac{\phi_0}{N} \sum_{l=1}^N (m_{n-1}-\langle x_l(t-\tau_1), v \rangle)+ \frac{\rho_0}{m}\sum_{l=1}^m  (m_{n-1}-\langle y_l(t-\tau_2), v \rangle) \\
&\leq 2K (M_{n-1}-m_{n-1}) - 2K \langle x_i(t)-x_j(t), v \rangle + \left(\phi_0+\rho_0\right)(-M_{n-1}+m_{n-1}).
\end{split}
\end{equation*}
For notational simplicity, we set
\[
\Lambda:=\min\{\psi_0, \phi_0,\rho_0\}.
\]
Then, the inequality simplifies to
\[
\frac{d}{d t} \langle x_i(t)-x_j(t), v \rangle \leq 2 \left( K-\Lambda\right) (M_{n-1}-m_{n-1}) - 2K \langle x_i(t)-x_j(t), v \rangle.
\]
Applying the Gr\"onwall's lemma on the time interval $[(n-1)\tau,t]$ with $t \in [(n-1)\tau,n\tau],$ we find that
\begin{equation*}
\begin{split}
\langle x_i(t)-x_j(t), v \rangle & \leq e^{-2K(t-n\tau+\tau)} \langle x_i(n\tau-\tau)-x_j(n\tau-\tau), v \rangle \\
& \quad + \left( 1- \frac{\Lambda}{K}\right)(M_{n-1}-m_{n-1}) (1-e^{-2K(t-n\tau+\tau)}).
\end{split}
\end{equation*}
Since this is valid for all $\ t \in [(n-1)\tau,n\tau], $ taking $t=n\tau$, we obtain
\begin{equation*} 
\begin{split} 
&\langle x_i(n\tau)-x_j(n\tau), v \rangle \cr
&\quad \leq e^{-2K\tau} \langle x_i(n\tau-\tau)-x_j(n\tau-\tau), v \rangle  + \left( 1-\frac{\Lambda}{K}\right)(M_{n-1}-m_{n-1}) (1-e^{-2K\tau}) \\
&\quad \leq e^{-2K\tau} |x_i(n\tau-\tau)-x_j(n\tau-\tau)| |v| + \left( 1-\frac{\Lambda}{K}\right)(M_{n-1}-m_{n-1}) (1-e^{-2K\tau}) \\
&\quad \leq D_{n-1} \left [e^{-2K\tau}+\left (1-\frac{\Lambda}{K}\right )(1-e^{-2K\tau}) \right ] \\
&\quad \leq D_{n-2} \left [1-\frac{\Lambda}{K}(1-e^{-2K\tau}) \right ],
\end{split}
\end{equation*}
where we used Remark \ref{2.3}. Consequently, we deduce that
$$d(n\tau)\leq D_{n-2} \left [1-\frac{\Lambda}{K}(1-e^{-2K\tau}) \right ].$$
Thus, we obtain the desired estimate for the case when the maximum diameter is determined by non-leader agents.

\medskip

\textbf{\emph{Case 2.}} Now, assume 
\[
d(n\tau)=|y_i(n\tau)-y_j(n\tau)|,
\] 
for some $i,j=1,\dots,m$. As in Case 1, we begin by normalizing the difference. Define
\[
v:=\frac{y_i(n\tau)-y_j(n\tau)}{|y_i(n\tau)-y_j(n\tau)|}.
\]
For $t \in [(n-1)\tau,n\tau]$, similarly as in Case 1, using the definition of the system, we write the time derivative of the projection difference between the leaders $y_i$ and $y_j$ as
\begin{equation*}
\begin{split}
  \frac{d}{d t} \langle y_i(t)-y_j(t), v \rangle  
&= \frac{1}{m} \sum_{l=1}^m \psi_{il}^{\tau_1}(t) \langle y_l(t-\tau_2)-y_i(t), v \rangle - \frac{1}{m}\sum_{l=1}^{m} \psi_{jl}^{\tau_1}(t) \langle y_l(t-\tau_2)-y_j(t),v \rangle\\
& = \frac{1}{m}\sum_{l=1}^m \psi_{il}^{\tau_1}(t) (\langle y_l(t-\tau_2), v \rangle - M_{n-1}) + \frac{1}{m}\sum_{l=1}^m \psi_{il}^{\tau_1}(t) (M_{n-1}-\langle y_i(t), v \rangle) \\
&\quad  + \frac{1}{m}\sum_{l=1}^m \psi_{jl}^{\tau_1}(t) (m_{n-1}-\langle y_l(t-\tau_2), v \rangle) + \frac{1}{m}\sum_{l=1}^m \psi_{jl}^{\tau_1}(t) (\langle y_j(t), v \rangle-m_{n-1})\cr
&\leq \frac{\psi_0}{m} \sum_{l=1}^m (\langle y_l(t-\tau_2), v \rangle - M_{n-1})+ K (M_{n-1}-\langle y_i(t), v \rangle)\cr
&\quad + \frac{\psi_0}{m} \sum_{l=1}^m (m_{n-1}-\langle y_l(t-\tau_2), v \rangle)+ K (\langle y_j(t), v \rangle - m_{n-1}).
\end{split}
\end{equation*}
This implies
\begin{equation*}
\begin{split}
  \frac{d}{d t} \langle y_i(t)-y_j(t), v \rangle  
&\leq K (M_{n-1}-m_{n-1}) - K \langle y_i(t)-y_j(t), v \rangle \\
&\quad + \frac{\psi_0}{m}\sum_{l=1}^{m}(\langle y_l(t-\tau_2), v \rangle - M_{n-1})  + \frac{\psi_0}{m} \sum_{l=1}^{m}(m_{n-1}-\langle y_l(t-\tau_2), v \rangle)\cr
&\leq K (M_{n-1}-m_{n-1}) - K \langle y_i(t)-y_j(t), v \rangle + \Lambda(-M_{n-1}+m_{n-1})\cr
&= \left(K-\Lambda\right)(M_{n-1}-m_{n-1}) - K \langle y_i(t)-y_j(t), v \rangle.
\end{split}
\end{equation*}
Applying the Gr\"onwall's lemma over the interval $[(n-1)\tau,t]$ with $t \in [(n-1)\tau,n\tau],$ we find that
\begin{equation*}
\begin{split}
\langle y_i(t)-y_j(t), v \rangle & \leq e^{-K(t-n\tau+\tau)} \langle y_i(n\tau-\tau)-y_j(n\tau-\tau), v \rangle \\
&\quad + \left( 1- \frac{\Lambda}{K}\right)(M_{n-1}-m_{n-1}) (1-e^{-K(t-n\tau+\tau)}).
\end{split}
\end{equation*}
Taking $t=n\tau$, this simplifies to
\[
\begin{split} 
& \langle y_i(n\tau)-y_j(n\tau), v \rangle  \cr
&\quad \leq e^{-K\tau} \langle y_i(n\tau-\tau)-y_j(n\tau-\tau), v \rangle + \left( 1-\frac{\Lambda}{K}\right)(M_{n-1}-m_{n-1}) (1-e^{-K\tau}) \\
&\quad \leq e^{-K\tau} |y_i(n\tau-\tau)-y_j(n\tau-\tau)| |v| + \left( 1-\frac{\Lambda}{K}\right)(M_{n-1}-m_{n-1}) (1-e^{-K\tau}) \\
&\quad \leq D_{n-1} \left [e^{-K\tau}+1-\frac{\Lambda}{K}(1-e^{-K\tau}) \right ] \\
&\quad \leq  D_{n-2} \left [1-\frac{\Lambda}{K}(1-e^{-K\tau}) \right ].
\end{split}
\]
Thus, we conclude that, 
\[
d(n\tau) \leq D_{n-2} \left [1-\frac{\Lambda}{K}(1-e^{-K\tau}) \right ].
\]
This completes the derivation for Case 2.

\medskip

\textbf{\em Case 3.} Now, assume that there exist indices $i\in\{1,\dots,N\}$ and $j\in\{1,\dots,m\}$ \[
d(n\tau)= \vert x_i(n\tau)-y_j(n\tau)\vert. 
\] 
In this mixed case, the maximum diameter is achieved by a non-leader and a leader. As before, we begin by normalizing the difference; define
\[
v:=\frac{x_i(n\tau)-y_j(n\tau)}{|x_i(n\tau)-y_j(n\tau)|}.
\]
Then, the distance can be expressed in the direction $v$ as 
\[
\vert x_i(n\tau)-y_j(n\tau)\vert =\langle x_i(n\tau)-y_j(n\tau), v \rangle. 
\]

For $t\in[(n-1)\tau,n\tau]$, by using almost the same arguments used in the previous cases, we deduce 
\begin{equation*}
\begin{split}
 \frac{d}{d t} \langle x_i(t)-y_j(t), v \rangle 
&\leq 2K (M_{n-1}-m_{n-1}) - 2K \langle x_i(t)-y_j(t), v \rangle \\
&\quad  + \frac{\Lambda}{N}  \sum_{l=1}^{N}(\langle x_l(t-\tau_2), v \rangle - M_{n-1})  + \frac{\Lambda}{m} \sum_{l=1}^{m}(\langle y_l(t-\tau_1), v \rangle - M_{n-1}) \\
&\quad + \frac{\Lambda}{m} \sum_{l=1}^{m} (m_{n-1}-\langle y_l(t-\tau_1), v \rangle)\\
& \leq \left(2K-\Lambda\right)(M_{n-1}-m_{n-1}) - 2K \langle x_i(t)-y_j(t), v \rangle,
\end{split}
\end{equation*}
where we used that for all $l=1, \dots, N, $
\[
\langle x_l(t-\tau_2), v \rangle - M_{n-1} \le 0.
\]
Again, analogously, we obtain
\[
\begin{split} 
d(n\tau) &\leq e^{-2K\tau} \langle x_i(n\tau-\tau)-y_j(n\tau-\tau), v \rangle  + \left( 1- \frac{\Lambda}{2K}\right)(M_{n-1}-m_{n-1}) (1-e^{-2K\tau}) \\
& \leq e^{-2K\tau} |x_i(n\tau-\tau)-y_j(n\tau-\tau)| |v| + \left( 1- \frac{\Lambda}{2K}\right)(M_{n-1}-m_{n-1}) (1-e^{-2K\tau}) \\
& \leq D_{n-1} \left [e^{-2K\tau}+\left (1-\frac{\Lambda}{2K}\right )(1-e^{-2K\tau}) \right ] \\
& \leq  D_{n-2} \left [1-\frac{\Lambda}{2K}(1-e^{-2K\tau}) \right ].
\end{split}
\]

Finally, to complete the proof of the lemma, we define
\begin{equation} \label{C}
C:=1-\frac{\Lambda}{2K}\Bigl(1-e^{-K\tau}\Bigr),
\end{equation}
which yields the desired estimate. 
\end{proof}

\begin{Remark}
From the proof of Lemma \ref{6}, it becomes evident that the normalization chosen for the weight functions in \eqref{weight} plays a crucial role in the dynamics. Roughly speaking, this normalization ensures that the influence exerted by the leaders constitutes half of the total influence on any non-leader. This balanced distribution of influence is essential for deriving the homogeneous contraction estimates that lead to consensus.
\end{Remark}

%
%
%
%
%
%
\subsection{Exponential consensus: Proof of Theorem \ref{decay_Th}}
Now, we prove the consensus result stated in Theorem \ref{decay_Th}.

Let $(x_i(t),y_j(t))$, with $i=1,\dots,N$ and $j=1,\dots,m$, be the solution of \eqref{eqA} with the initial conditions \eqref{ic1A} and \eqref{ic2A}. Our goal is to show that the diameter of the system decays exponentially. To achieve this, we first claim that there exists a constant $\tilde{C}\in (0,1)$ such that
\begin{equation} \label{claim1}
D_{n+1} \leq \tilde{C} D_{n-2}, \quad \forall \, n \geq 2.
\end{equation}

We start by observing that from Lemma \ref{5} the following estimate can be deduced:
\[
D_{n+1} \leq e^{-2K\tau}d(n\tau)+(1-e^{-2K\tau})D_n.
\]
To illustrate this, assume that for some $n\in\mathbb{N}_0$, there exist $s,t\in[n\tau,n\tau+\tau]$ and indices $i,j\in\{1,\dots,N\}$ such that
\[
D_{n+1}=|x_i(s)-x_j(t)|.
\]
Assume that $|x_i(s)-x_j(t)|>0,$ since the case $|x_i(s)-x_j(t)|=0$ is trivial. Define the unit vector
\[
v:=\frac{x_i(s)-x_j(t)}{|x_i(s)-x_j(t)|}.
\]
Then, by definition,
\[
D_{n+1}=\langle x_i(s)-x_j(t),v \rangle.
\]
Using \eqref{eq21} with $t_0=n\tau,$ we obtain 
\begin{equation}\label{eq8}
\begin{split}
\langle x_i(s),v \rangle & \leq e^{-2K(s-n\tau)}\langle x_i(n\tau),v \rangle +(1-e^{-2K(s-n\tau)})M_n \\
& = e^{-2K(s-n\tau)}(\langle x_i(n\tau),v \rangle - M_n) + M_n \\
& \leq e^{-2K\tau}\langle x_i(n\tau),v \rangle +(1-e^{-2K\tau})M_n.
\end{split}
\end{equation}
Analogously, using \eqref{eq22} we have 
\begin{equation} \label{eq9}
\begin{split}
\langle x_j(t),v \rangle
& \geq e^{-2K\tau}\langle x_j(n\tau),v \rangle +(1-e^{-2K\tau})m_n.
\end{split}
\end{equation}
Subtracting \eqref{eq9} from \eqref{eq8} yields
\begin{equation*}
\begin{split}
D_{n+1} & \leq e^{-2K\tau} \langle x_i(n\tau)-x_j(n\tau),v \rangle + (1-e^{-2K\tau})D_n \\
& \leq e^{-2K\tau}d(n\tau)+(1-e^{-2K\tau})D_n.
\end{split}
\end{equation*}
A similar reasoning applies if $D_{n+1}$ is defined by the other two possible forms (involving leader-leader or leader-non-leader differences).

From this and Lemma \ref{6}, we get that
\begin{equation*}
\begin{split}
D_{n+1} & \leq e^{-2K\tau}d(n\tau)+(1-e^{-2K\tau})D_n \\
& \leq e^{-2K\tau}C D_{n-2}+(1-e^{-2K\tau})D_n\\
& \leq e^{-2K\tau}C D_{n-2}+(1-e^{-2K\tau})D_{n-2} \\
& = \left( 1-e^{-2K\tau}(1-C)\right)D_{n-2},
\end{split}
\end{equation*}
where $C$ is defined in \eqref{C}. Thus, setting 
$$ \tilde{C}:=1-e^{-2K\tau}\frac{\Lambda}{2K}(1-e^{-K\tau}),$$
we obtain the claim \eqref{claim1}.

This recursive inequality implies that
\begin{equation} \label{ineq}
D_{3n}\leq \tilde{C}^n D_0, \quad \forall \, n \geq 1.
\end{equation}
Since $\tilde{C}\in (0,1)$, we can rewrite \eqref{ineq} as
\[
D_{3n} \leq e^{-3n\gamma \tau}D_0, 
\]
where 
\[
\gamma:=\frac{1}{3\tau}\ln \left(\frac{1}{\tilde{C}}\right) = -\frac{1}{3\tau}\ln \left( 1-e^{-2K\tau}\frac{\Lambda}{2K}(1-e^{-K\tau})\right).
\]

Finally,  fix any time $t\geq 0$. Then there exists  $n\in\mathbb{N}_0$ such that $t\in [3n\tau-\tau,3n\tau+2\tau]$. By Lemma \ref{3} and the definition of the global diameter \eqref{diameter}, we have
\[
d(t)\leq D_{3n}\leq e^{-3n\gamma\tau}D_0.
\]
Since $t\leq 3n\tau+2\tau$, it follows that
\[
d(t)\leq e^{-\gamma(t-2\tau)}D_0.
\]
This completes the proof of Theorem \ref{decay_Th}.

%
%
%
%
%
%
\section{Global existence of measure-valued solutions of the mean-field models}\label{sec3}
In this section, we establish the global-in-time existence and uniqueness of measure-valued solutions to the mean-field systems \eqref{pde1} and \eqref{pde2}, which arise as formal limits of the particle system \eqref{eqA} when the number of agents tends to infinity.  

We begin with the analysis of the first model \eqref{pde1}, which describes a system with a finite number of leaders interacting with a continuum of followers.
%
%
%
%
%
%

\subsection{Few leaders and many followers system}

We consider the mean-field equation \eqref{pde1}, which models the collective behavior of a large population of followers influenced by a finite number of leaders. The leaders follow prescribed trajectories $\{\bar{y}_j(t)\}_{j=1}^m$, while the followers evolve according to a transport equation driven by the interaction terms. The velocity field $v_t^m(x)$ is defined by the interaction kernel \eqref{flux1}, which combines leader-follower and follower-follower interactions.

To ensure the well-posedness of the transport equation, we require that the influence functions $\phi$ and $\rho$ satisfy aforenoted regularity and boundedness conditions. We denoted by $L_{\phi}$ and $L_{\rho}$ the Lipschitz constants of $\phi$ and $\rho$, respectively. We also write $\bar\rho_j^{\tau_1}(x) := \rho(x, \bar{y}_j(t - \tau_1))$ for the interaction term between the followers and the $j$-th leader.

We now prove that the velocity field $v_t^m(x)$ is globally Lipschitz and bounded under the assumption that the follower density $\nu_t$ has compact support.

\begin{Lemma}\label{lip}
Let $\nu_t \in C([0,T); \mathcal{P}(\mathbb{R}^d))$ be a family of probability measures with compact support, i.e.,
\[
\supp \nu_t \subset B^d(0,R), \quad \forall t \in [0,T],
\]
where $B^d(0,R)$ denotes the ball of radius $R>0$ centered at the origin in $\mathbb{R}^d$. Then the velocity field $v_t^m(x)$ defined by \eqref{flux1} satisfies the following properties:
\begin{itemize}
    \item[(i)] (Lipschitz continuity) There exists a constant $\tilde{K}>0$ such that
    \begin{equation}\label{diff_flux}
    |v_t^m(x) - v_t^m(\tilde{x})| \le \tilde{K}|x - \tilde{x}|, \quad \forall \, x, \tilde{x} \in B^d(0,R), \; t \in [0,T].
    \end{equation}
    \item[(ii)] (Uniform boundedness) There exists a constant $\tilde{C}>0$ such that
    \begin{equation}\label{norm_flux}
    |v_t^m(x)| \le \tilde{C}, \quad \forall \,x \in B^d(0,R), \; t \in [0,T].
    \end{equation}
\end{itemize}
\end{Lemma}
\begin{proof}
Fix $x, \tilde{x} \in B^d(0,R)$. We split the difference $|v_t^m(x) - v_t^m(\tilde{x})|$ into two parts:
\[
\begin{split}
\vert v_t^m(x)-v_t^m(\tilde{x})\vert & \leq \Big\vert \int_{\mathbb{R}^d} \phi(x,y)(y-x)\nu_{t-\tau_2}(dy)-\int_{\mathbb{R}^d}\phi(\tilde{x},y)(y-\tilde{x})\nu_{t-\tau_2}(dy)\Big\vert \\
& \quad + \Big\vert \frac{1}{m} \sum_{j=1}^m \bar\rho_j^{\tau_1}(x)(\bar y_j(t-\tau_1)-x)-\frac{1}{m} \sum_{j=1}^m \bar\rho_j^{\tau_1}(\tilde{x})(\bar y_j(t-\tau_1)-\tilde{x})\Big\vert \\
& =: I + II.
\end{split}
\]
Using the Lipschitz continuity of $\phi(x,y)$ and the compact support of the measure $\nu_t,$ we find that 
\[
\begin{split}
I & \leq \Big\vert \int_{\mathbb{R}^d} \left( \phi(x,y)-\phi(\tilde{x},y)\right) y \nu_{t-\tau_2}(dy)\Big\vert
 + \Big\vert \int_{\mathbb{R}^d}\phi(x,y)x \nu_{t-\tau_2}(dy)-\int_{\mathbb{R}^d}\phi(\tilde{x},y)\tilde{x}\nu_{t-\tau_2}(dy)\Big\vert \\
 & \leq RL_{\phi}\vert x-\tilde{x}\vert + \Big\vert \int_{\mathbb{R}^d} \left(\phi(x,y)-\phi(\tilde{x},y)\right)x \nu_{t-\tau_2}(dy)\Big\vert 
  + \Big\vert \int_{\mathbb{R}^d} \phi(\tilde{x},y)(x-\tilde{x})\nu_{t-\tau_2}(dy) \Big\vert \\
  & \leq (K+2RL_{\phi})\vert x-\tilde{x}\vert.
\end{split}
\]
Similarly, we estimate $II$ as
\[
\begin{split}
II & \leq \frac{1}{m} \sum_{j=1}^m \vert \bar\rho_j^{\tau_1}(x)-\bar\rho_j^{\tau_1}(\tilde{x}) \vert \vert \bar y_j(t-\tau_1) \vert + \frac{1}{m} \sum_{j=1}^m \vert \bar\rho_j^{\tau_1}(x) - \bar\rho_j^{\tau_1}(\tilde{x}) \vert \vert x \vert   + \frac{1}{m} \sum_{j=1}^m \vert \bar\rho_j^{\tau_1}(\tilde{x}) \vert \vert x-\tilde{x} \vert \\
& \leq [L_{\rho}(C_0+R)+K] \vert x - \tilde{x} \vert,
\end{split}
\]
where $C_0 > 0$ is the bound on leader trajectories from Lemma \ref{4}. Combining the bounds yields \eqref{diff_flux} with $\tilde{K} := 2RL_\phi + 2K + L_\rho(C_0 + R)$.

To prove \eqref{norm_flux}, we estimate directly:
\[
\begin{split}
\vert v_t^m(x)\vert & \leq \Big\vert \int_{\mathbb{R}^d} \phi(x,y)(y-x)\nu_{t-\tau_2}(dy)\Big\vert + \frac{1}{m} \Big\vert \sum_{j=1}^m \bar\rho_j^{\tau_1}(t)(\bar y_j(t-\tau_1)-x) \Big\vert \\
& \leq 2KR + K(C_0+R),
\end{split}
\]
which gives \eqref{norm_flux} with $\tilde{C} := K(3R + C_0)$.
\end{proof}

Now, we are in a position to prove the global existence and uniqueness of solutions stated in Theorem~\ref{1.7} for the mean-field system \eqref{pde1}.

\begin{proof}[Proof of Theorem \ref{1.7}: existence and uniqueness in Case (i)]
We begin by considering the ODE system describing the evolution of the leaders $\{y_j(t)\}_{j=1}^m$. Since the interaction functions are Lipschitz continuous and bounded, standard results from the theory of delay differential equations (see, e.g., \cite{Halanay, Hale}) ensure existence and uniqueness of solutions. Specifically, by applying the Banach fixed-point theorem on small time intervals and iterating the solution step-by-step, we can construct a unique global-in-time solution for the leader dynamics.

We now turn to the second component of the system \eqref{pde1}, which is a continuity equation driven by a delayed, nonlocal velocity field. To obtain local-in-time existence and uniqueness of measure-valued solutions, we apply Lemma \ref{lip}, which provides Lipschitz and boundedness estimates on the velocity field $v_t^m(x)$, together with \cite[Theorem 3.10]{Carrillo}, which guarantees well-posedness under such conditions, provided that the solution remains compactly supported.

Thus, to extend this local-in-time solution to a global one, it is sufficient to control the growth of the support of $\nu_t$. We do this by estimating the maximal spatial extension of the support. First, since the leaders' dynamics do not directly depend on the follower particles $\{x_i\}_{i=1}^N$, we focus on the leader dynamics and define the bound
\[
C_0^y := \max_{s \in [-\tau,0]} \max_{j=1,\dots,m} |y_j(s)|.
\]
By Lemma \ref{4}, we then have the uniform-in-time bound
\[
|y_j(t)| \leq C_0^y, \quad \text{for all } t \geq 0 \, \text{ and } j = 1,\dots,m.
\]

Next, define the maximal radius of the support of the measure \( \nu_t \) as
\begin{equation}\label{R}
R_X(t) := \max \left\{ \max_{s \in [-\tau,t]} \sup_{x \in \overline{\supp \nu_s}} |x|, \; C_0^y \right\}.
\end{equation}
We now perform a continuity argument to control $R_X(t)$. Similarly to Lemma \ref{b}, for a fixed $\epsilon > 0$, we define a set
\[
 \mathcal{T}^{\epsilon}:= \left\{ t >0 : R_X(s) < R_X(0)+\epsilon, \ \forall s \in [0,t) \right\}.
 \]
By continuity of trajectories,  $\mathcal{T}^{\epsilon}$ is nonempty. Let $T_{\epsilon}:= \sup \mathcal{T}^{\epsilon}$. Our goal is to show that  $T_\epsilon \ge \tau^*$, where $\tau^* := \min\{\tau_1, \tau_2\}$. Suppose, for contradiction, that $T_{\epsilon} < \tau^*$. Then, we find
\begin{equation}\label{A}
\lim_{t \rightarrow T_{\epsilon}^-} R_X(t) = R_X(0)+\epsilon
\end{equation}
and 
\[
R_X(t) < R_X(0) + \epsilon, \quad \ \forall \ t < T_{\epsilon}
\]
Consider the system of characteristics $X(t;x):[0,T_{\epsilon}] \times \mathbb{R}^d \rightarrow \mathbb{R}^d$ associated to the continuity equation in \eqref{pde1}, given by 
\[
\begin{cases}
\frac{d}{d t} X(t;x)=v_t(X(t;x)), \\
X(0;x)=x
\end{cases}
\]
for $x \in \mathbb{R}^d.$ Then, applying Lemma \ref{lip}, this system admits a unique solution on the time interval $[0,T_{\epsilon}]$. Note that the measure-valued solution is transported by the characteristic flow, namely
\[
\nu_t = X(t;\cdot)\#\nu_0 , \qquad t\in[0,T_{\epsilon}].
\]
In particular, if $x\in\supp\nu_0$ then $X(t;x)\in\supp\nu_t$ for all $t\in[0,T_{\epsilon}]$.

Let us simplify notation by writing $X(t;x)$ as $X(t)$ and denote $\bar \rho_j^{\tau_1}(X(t)) := \rho(X(t), \bar y_j(t-\tau_1))$. Then, we estimate
\begin{equation*}
\begin{split}
\frac{1}{2} \frac{d}{d t} \vert X(t)\vert^2 & = \langle \dot{X(t)}, X(t) \rangle \\
& = \int_{\mathbb{R}^d} \phi(X(t),y)\langle y-X(t),X(t)\rangle\nu_{t-\tau_2}(dy)  +\frac{1}{m} \sum_{j=1}^m \bar\rho_j^{\tau_1}(X(t)) \langle \bar y_j(t-\tau_1)-X(t),X(t) \rangle \\
& = \int_{\mathbb{R}^d} \phi(X(t),y)\langle y,X(t)\rangle\nu_{t-\tau_2}(dy) - \int_{\mathbb{R}^d} \phi(X(t),y)\vert X(t)\vert^2 \nu_{t-\tau_2}(dy) \\
& \quad + \frac{1}{m} \sum_{j=1}^m \bar\rho_j^{\tau_1}(X(t)) \left( \langle \bar y_j(t-\tau_1),X(t) \rangle - \vert X(t) \vert^2 \right) 
\end{split}
\end{equation*}
Using the definition \eqref{R} and the fact that \( |y|, |\bar y_j| \leq R_X(t) \), we deduce
\begin{align*}
&\frac{1}{2} \frac{d}{d t} \vert X(t)\vert^2  \cr
&\quad \leq  \vert X(t)\vert \left( \int_{\mathbb{R}^d} \phi(X(t),y) (R_X(t)- \vert X(t) \vert) \nu_{t-\tau_2}(dy) +\frac{1}{m} \sum_{j=1}^m \bar\rho_j^{\tau_1}(X(t))(R_X(t)-\vert X(t)\vert \right).
\end{align*}
Since $R_X(t) - |X(t)| \geq 0$ for $t < T_{\epsilon}$, and $\phi$ and $\bar\rho_j^{\tau_1}$ are bounded by \( K \), we find
\[
\frac{d}{dt} |X(t)| \leq 4K (R_X(0) + \epsilon - |X(t)|).
\]
By Gr\"onwall's inequality, this implies that $|X(t)| < R_X(0) + \epsilon$ on $[0, T_{\epsilon}]$, contradicting \eqref{A}. Thus,  $T_{\epsilon} \geq \tau^*$. Since $\epsilon > 0$ was arbitrary, we conclude that the support of \( \nu_t \) remains uniformly bounded on \( [0, \tau^*] \), and the solution can be extended beyond \( \tau^* \).
Repeating this argument iteratively on time intervals of length \( \tau^* \), we construct a unique global-in-time solution. 

Finally, as recalled above, following \cite{Carrillo}, we obtain that the measure-valued solution satisfies the weak formulation \eqref{measure}, and that the corresponding push-forward relation \eqref{pushf} holds.
\end{proof}
%
%
%
%
%
%
\subsection{Infinite population limit for both leaders and followers}

We now consider the case in which both populations, leaders and followers, consist of infinitely many agents. For the mean-field system \eqref{pde2}, we establish the existence and uniqueness of measure-valued solutions, using arguments analogous to those employed in the previous subsection. 

As before, we assume that the interaction kernels $\psi(x,y)$, $\phi(x,y)$, and $\rho(x,y)$ appearing in the fluxes \eqref{flux2} and \eqref{flux3} are positive, bounded, and Lipschitz continuous. Let us denote by
$$ L := \max\{L_{\psi}, L_{\phi}, L_{\rho}\}, $$
where $L_{\psi}, L_{\phi}, L_{\rho}$ are the respective Lipschitz constants.

We begin with a regularity estimate on the velocity fields induced by the measure solutions.

\begin{Lemma}\label{lip2}
Consider the system $\eqref{pde2}$ subject to the initial data $\eqref{init2}$. Given a time $T > 0$, suppose that $\bar\mu, \bar\nu \in C([0,T); \mathcal{P}_{\infty}(\mathbb{R}^d))$ are measures with uniformly compact supports:
\[
\supp \bar\mu_t \subset B^d(0,R_1), \quad \supp \bar\nu_t \subset B^d(0,R_2), \quad \forall\, t \in [0,T), 
\]
where $B^d(0,R_i)$ denotes the ball of radius $R_i > 0$ centered at the origin in $\mathbb{R}^d$ for $i = 1, 2$. 

Then, the velocity fields $\bar u_t$ and $\bar v_t$ defined in $\eqref{flux2}$--$\eqref{flux3}$ satisfy the following estimates:
\begin{itemize}
  \item (Lipschitz continuity) There exist constants $K_1, K_2 > 0$ such that
  $$
  |\bar u_t(x) - \bar u_t(\tilde{x})| \leq K_1 |x - \tilde{x}|, \quad |\bar v_t(z) - \bar v_t(\tilde{z})| \leq K_2 |z - \tilde{z}|,
  $$
  for all $x, \tilde{x} \in B^d(0,R_1)$, $z, \tilde{z} \in B^d(0,R_2)$, and $t \in [0,T]$.

  \item (Uniform boundedness) There exist constants $C_1, C_2 > 0$ such that
  $$
  |\bar u_t(x)| \leq C_1, \quad |\bar v_t(z)| \leq C_2,
  $$
  for all $x \in B^d(0,R_1)$, $z \in B^d(0,R_2)$, and $t \in [0,T]$.
\end{itemize}
\end{Lemma}

\begin{proof}
Arguing in the same way as in the proof of Lemma $\ref{lip}$, we estimate the differences in the velocity fields with constants:
\[
K_1 := K + 2R_1 L, \qquad K_2 = 2K+L(R_1+3R_2).
\]
The uniform boundedness of the velocity fields follows similarly with constants:
\[
C_1 := 2K R_1, \qquad C_2= K (R_1+3R_2).
\]
This completes the proof.
\end{proof}

With the above estimates in hand, we now prove the existence and uniqueness result for the mean-field system \eqref{pde2}.

\begin{proof}[Proof of Theorem \ref{1.7}: existence and uniqueness in Case (ii)]
We proceed in parallel to the argument for Case (i) in the previous subsection. Applying Lemma $\ref{lip2}$ and the framework established in \cite{Carrillo}, we first deduce the global-in-time existence and uniqueness of a measure-valued solution $\bar\mu \in C([0,T); \mathcal{P}_\infty(\mathbb{R}^d))$.

We then define the support control quantity
\[
\tilde{R}_X(t) := \max_{-\tau \leq s \leq t} \left\{ \max_{x \in \overline{\supp \bar\mu_s}} |x|, \max_{z \in \overline{\supp \bar\nu_s}} |z| \right\},
\]
and show that $\tilde{R}_X(t)$ remains uniformly bounded in time. The regularity of the velocity fields and compact support of the initial data then yield global-in-time existence and uniqueness of the second component $\bar\nu \in C([0,T); \mathcal{P}_\infty(\mathbb{R}^d))$, thereby completing the proof.
\end{proof}

%
%
%
%
%
%
\section{Stability and consensus in the mean-field regime}\label{sec4}
 In this section, we analyze the stability and asymptotic behavior of solutions to the mean-field systems introduced in Section \ref{sec3}. We begin by establishing a stability estimate with respect to initial data, measured in the Wasserstein distance. This will be followed by an investigation of large-time consensus behavior.

%
%
%
%
%
%
\subsection{Wasserstein stability estimate}

We establish a continuous dependence result for solutions of the mean-field model \eqref{pde1} in terms of their initial data, using the Wasserstein distance framework. We begin by recalling the definition of the Wasserstein distance.

\begin{Definition}
Let $\nu_t^1,\nu_t^2 \in \mathcal{P}(\mathbb{R}^d)$ two probability measures in $\mathbb{R}^d.$ We define the Wasserstein distance of order $1 \leq p < \infty$ between $\nu_t^1$ and $\nu_t^2$ the quantity 
$$ d_p(\nu_t^1,\nu_t^2):= \inf _{\pi \in \Pi(\nu_t^1,\nu_t^2)} \left( \int_{\mathbb{R}^d \times \mathbb{R}^d} \vert x-y \vert \pi(d x,d y) \right)^{\frac{1}{p}},$$
and for $p=\infty,$ limiting case $p \rightarrow +\infty,$ 
$$ d_{\infty}(\nu_t^1,\nu_t^2):= \inf _{\pi \in \Pi(\nu_t^1,\nu_t^2)} \left\{ \sup_{(x,y)\in \supp(\pi)}\vert x-y \vert \right\}. $$ 
\end{Definition}

The following lemma provides a stability estimate for solutions of \eqref{pde1} with respect to initial perturbations.

\begin{Lemma}\label{4.3}
Let $(\bar y^1,\nu_t^1)$ and $(\bar y^2, \nu_t^2)$ be two solutions of \eqref{pde1} constructed in Theorem \ref{1.7}, corresponding to initial data $(\bar y^{1,0}, g^1)$ and $(\bar y^{2,0}, g^2)$, respectively. Then, there exists a constant $\tilde{C}>0$, depending on $T$ but independent of $p$, such that for all $t \in [0,T]$ and $p \in [1,\infty)$,
\begin{align*}
&d_p(\nu_t^1,\nu_t^2) + \left( \frac1m \sum_{j=1}^m |\bar y_j^1(t) - \bar y_j^2(t)|^p\right)^\frac1p \cr
&\quad \leq \tilde C \sup_{s \in [-\tau,0]}d_p(g_s^1,g_s^2) + \tilde C \sup_{s \in [-\tau,0]}\left( \frac1m \sum_{j=1}^m |\bar y_j^{1,0}(s) - \bar y_j^{2,0}(s)|^p\right)^\frac1p.
\end{align*}
In the case $p=\infty$, the following bound holds:
\[
d_\infty(\nu_t^1,\nu_t^2) +  \max_{i=1,\dots, m} |\bar y_i^1(t) - \bar y_i^2(t)|  \leq \tilde C \sup_{s \in [-\tau,0]}d_\infty(g_s^1,g_s^2) + \tilde C \sup_{s \in [-\tau,0]}\max_{i=1,\dots, m} |\bar y_i^{1,0}(s) - \bar y_i^{2,0}(s)|.
\]
\end{Lemma}
\begin{proof}
We begin by constructing the system of characteristics associated with each solution. For $i = 1,2$, define $X^i(t;x): [0,T] \times \mathbb{R}^d \to \mathbb{R}^d$ as the flow map solving
$$
\begin{cases}
\displaystyle \frac{d}{d t}X^i(t;x)= v_t^{m,i}(X^i(t;x)), \quad  x \in \mathbb{R}^d,\\[2mm]
X^i(0;x)=x,
\end{cases}
$$
where $v_t^{m,i}$ is given by the formula \eqref{flux1}. By Theorem \ref{1.7}, the flows $X^i$ are well-defined on the interval $[0,T]$. By standard properties of transport by characteristics, we have $\nu_t^i = X^i(t;\cdot) \# \nu_0^i,$ for all $t\in [0,T]$ and $i=1,2$. As before, we define the quantity $R_X^i(t)$ as in \eqref{R}. Let $\mathcal{S}_0: \mathbb{R}^d \to \mathbb{R}^d$ be the optimal transport map pushing $\nu_0^1$ to $\nu_0^2$ with respect to the $p$-Wasserstein distance, i.e.,
\[
\nu_0^2 = \mathcal{S}_0 \# \nu_0^1, \quad d_p(\nu_0^1,\nu_0^2) = \displaystyle \left(\int_{\mathbb{R}^d} \vert x-\mathcal{S}_0(x)\vert^p \nu_0^1(dx)\right)^\frac{1}{p}.
\]
Therefore, defining a map 
$\mathcal{T}^t:= X^2(t;\cdot) \circ \mathcal{S}_0 \circ (X^1(t; \cdot))^{-1},$ for $t \in [0,T]$, we have
\begin{equation}\label{T}
\mathcal{T}^t \# \nu_t^1 = \nu_t^2,
\end{equation} 
and thus
\[
d_p(\nu_t^1,\nu_t^2) \leq \left( \int_{\mathbb{R}^d} \vert x-\mathcal{T}^t(x) \vert ^p \nu_t^1(dx) \right)^\frac{1}{p} =: \theta_p(t).
\]
Using the identity $\mathcal{T}^t \circ X^1(t; \cdot)= X^2(t;\cdot)\circ \mathcal{S}_0$, we rewrite $\theta_p(t)$ as
\[
\theta_p(t) = \left(\int_{\mathbb{R}^d} \vert X^1(t;x)-X^2(t;\mathcal{S}_0(x))\vert^p \nu_0^1(dx)\right)^\frac1p.
\]
To incorporate the time-delay structure, we extend $\mathcal{T}^s$ on $[-\tau, 0]$ as the optimal transport map between $g_s^1$ and $g_s^2$, and define
\[
\theta_p(s):=d_p(g_s^1,g_s^2)=\left(\int_{\mathbb{R}^d}\vert x-\mathcal{T}^s(x)\vert^p g_s^1(dx)\right)^\frac{1}{p}, \quad s \in [-\tau,0].
\]

Next, we estimate the time derivative of $\theta_p(t)^p$. For $t \in (0,T)$,
\begin{equation*} \begin{split}
\frac{d}{d t} \theta_p(t)^p & \leq p \int_{\mathbb{R}^d} \vert X^1(t;x)-X^2(t;\mathcal{S}_0(x))\vert^{p-1}   \vert v_s^{m,1}(X^1(t;x))-v_s^{m,2}(X^2(t;\mathcal{S}_0(x)))\vert \nu_0^1(dx) \\
& =p \int_{\mathbb{R}^d} \vert x-\mathcal{T}^t(x)\vert^{p-1} \vert v_t^{m,1}(x)-v_t^{m,2}(\mathcal{T}^t(x))\vert \nu_t^1(dx)\cr
&\leq p \theta_p(t)^{p-1}\left( \int_{\R^d} |v_t^{m,1}(x)-v_t^{m,2}(\mathcal{T}^t(x))|^p \nu_t^1(dx)\right)^\frac1p,
\end{split} 
\end{equation*}
and thus
\[
\frac{d}{dt}\theta_p(t) \leq \left( \int_{\R^d} |v_t^{m,1}(x)-v_t^{m,2}(\mathcal{T}^t(x))|^p \nu_t^1(dx)\right)^\frac1p.
\]
Now, we estimate the velocity difference as
\begin{align*}
&v_t^{m,1}(x)-v_t^{m,2}(\mathcal{T}^t(x))\cr
&\quad = \int_{\R^d} \phi(x,y)(y-x) \nu^1_{t-\tau_2}(dy) - \int_{\R^d} \phi(\mathcal{T}^t(x),y)(y-\mathcal{T}^t(x)) \nu^2_{t-\tau_2}(dy)\cr
&\qquad + \frac1m \sum_{j=1}^m \rho(x, \bar y_j^1(t-\tau_1))(\bar y_j^1(t-\tau_1) - x) - \frac1m \sum_{j=1}^m \rho(\mathcal{T}^t(x), \bar y_j^2(t-\tau_1))(\bar y_j^2(t-\tau_1) - \mathcal{T}^t(x))\cr
&\quad =: I + II.
\end{align*}
Here, using \eqref{T}, 
\begin{align*}
I&= \int_{\R^d} \phi(x,y)(y-x) \nu^1_{t-\tau_2}(dy) - \int_{\R^d} \phi(\mathcal{T}^t(x),\mathcal{T}^{t-\tau_2}(y))(\mathcal{T}^{t-\tau_2}(y)-\mathcal{T}^t(x)) \nu^1_{t-\tau_2}(dy)\cr
&= \int_{\R^d} \big( \phi(x,y) - \phi(\mathcal{T}^t(x),\mathcal{T}^{t-\tau_2}(y))\big) (y-x) \nu^1_{t-\tau_2}(dy) \cr
&\quad   + \int_{\R^d} \phi(\mathcal{T}^t(x),\mathcal{T}^{t-\tau_2}(y))((y-x) - (\mathcal{T}^{t-\tau_2}(y)-\mathcal{T}^t(x))) \nu^1_{t-\tau_2}(dy).
\end{align*}
Using the Lipschitz continuity and boundedness of the function $\phi$, we find
\begin{equation*}
\begin{split}
& \int_{\mathbb{R}^d}\vert \phi(x,y)-\phi(\mathcal{T}^t(x),\mathcal{T}^{t-\tau_2}(y))\vert \vert y-x\vert \nu^1_{t-\tau_2}(dy) \\
&\quad \leq L_\phi (\vert x \vert + R_X^1(t))\left[ \vert x-\mathcal{T}^t(x)\vert + \int_{\mathbb{R}^d} \vert y-\mathcal{T}^{t-\tau_2}(y)\vert \nu^1_{t-\tau_2}(dy)\right]
\end{split}
\end{equation*}
and
\begin{equation*}
\begin{split}
& \int_{\mathbb{R}^d}\vert \phi(\mathcal{T}^t(x),\mathcal{T}^{t-\tau_2}(y))\vert \vert y-x-(\mathcal{T}^{t-\tau_2}(y)-\mathcal{T}^t(x))\vert \nu^1_{t-\tau_2}(dy) \\
&\quad \leq K \vert x-\mathcal{T}^t(x)\vert + K \int_{\mathbb{R}^d} \vert y-\mathcal{T}^{t-\tau_2}(y)\vert \nu^1_{t-\tau_2}(dy).
\end{split}
\end{equation*}
This implies
\[
|I| \leq \big( L_\phi (\vert x \vert + R_X^1(t)) + K\big) \left[ \vert x-\mathcal{T}^t(x)\vert + \int_{\mathbb{R}^d} \vert y-\mathcal{T}^{t-\tau_2}(y)\vert \nu^1_{t-\tau_2}(dy)\right].
\]
Similarly, we estimate
\begin{align*}
|II| &\leq  \frac1m \sum_{j=1}^m \big| \rho(x, \bar y_j^1(t-\tau_1)) -  \rho(\mathcal{T}^t(x), \bar y_j^2(t-\tau_1))\big| |\bar y_j^1(t-\tau_1) - x|\cr
&\quad  + \frac1m \sum_{j=1}^m    \rho(\mathcal{T}^t(x), \bar y_j^2(t-\tau_1)) \big|(\bar y_j^1(t-\tau_1) - x) - (\bar y_j^2(t-\tau_1) - \mathcal{T}^t(x))\big|\cr
& \leq \big(L_\rho(C_0 + |x|) + K \big) \left[ |x-\mathcal{T}^t(x)| + \frac1m \sum_{j=1}^m |\bar y_j^1(t-\tau_1) - \bar y_j^2(t-\tau_1)|  \right].
\end{align*}
Combining the estimates for $I$ and $II$, we deduce
\begin{align*}
&|v_t^{m,1}(x)-v_t^{m,2}(\mathcal{T}^t(x))|\cr
&\quad \leq C(1 + |x|)\left[ |x-\mathcal{T}^t(x)| + \int_{\mathbb{R}^d} \vert y-\mathcal{T}^{t-\tau_2}(y)\vert \nu^1_{t-\tau_2}(dy) + \frac1m \sum_{j=1}^m |\bar y_j^1(t-\tau_1) - \bar y_j^2(t-\tau_1)|  \right]
\end{align*}
for some constant $C>0$  independent of $p$. Using H\"older's inequality, we get
\[
\int_{\mathbb{R}^d} \vert y-\mathcal{T}^{t-\tau_2}(y)\vert \nu^1_{t-\tau_2}(dy) \leq \left(\int_{\mathbb{R}^d} \vert y-\mathcal{T}^{t-\tau_2}(y)\vert^p \nu^1_{t-\tau_2}(dy)\right)^\frac1p
\]
and
\[
\frac1m \sum_{j=1}^m |\bar y_j^1(t-\tau_1) - \bar y_j^2(t-\tau_1)| \leq \left( \frac1m \sum_{j=1}^m |\bar y_j^1(t-\tau_1) - \bar y_j^2(t-\tau_1)|^p\right)^\frac1p =: \xi_p(t-\tau_1).
\]
This, together with the boundedness of support of $\nu^1_t$, yields
\[
\frac{d}{dt}\theta_p(t) \leq  C\theta_p(t) + C\theta_p(t-\tau_2) + C\xi_p(t-\tau_1)
\]
for some constant $C>0$ independent of $p$.  
To close the estimate, we estimate $\xi_p(t)$ using the leader dynamics. We first get
\begin{align*}
&\left|\frac{1}{m} \sum_{j=1}^m \psi(\bar y_i^1(t), \bar y_j^1(t-\tau_1))(\bar y_j^1(t-\tau_1) - \bar y_i^1(t)) - \frac{1}{m} \sum_{j=1}^m \psi(\bar y_i^2(t), \bar y_j^2(t-\tau_1))(\bar y_j^2(t-\tau_1) - \bar y_i^2(t))\right|\cr
&\quad \leq \frac{1}{m} \sum_{j=1}^m \big| \psi(\bar y_i^1(t), \bar y_j^1(t-\tau_1)) -  \psi(\bar y_i^2(t), \bar y_j^2(t-\tau_1)) \big| |\bar y_j^1(t-\tau_1) - \bar y_i^1(t)| \cr
&\qquad + \frac{1}{m} \sum_{j=1}^m \psi(\bar y_i^2(t), \bar y_j^2(t-\tau_1)) \big| (\bar y_j^1(t-\tau_1) - \bar y_i^1(t)) - (\bar y_j^2(t-\tau_1) - \bar y_i^2(t))\big|\cr
&\quad \leq \big(2C_0^y L_\psi + K\big) \left[ |\bar y_i^1(t) - \bar y_i^2(t)| + \left(\frac{1}{m} \sum_{j=1}^m| \bar y_j^1(t-\tau_1) - \bar y_j^2(t-\tau_1)|^p\right)^\frac1p \right].
\end{align*}
This yields
\[
\frac{d}{dt} \xi_p(t) \leq C\xi_p(t) + C\xi_p(t-\tau_1),
\]
where $C>0$ is a constant independent of $p$. Following arguments in \cite{CH, CPP}, from the above, we obtain
\[
\xi_p(t) \leq  C\sup_{s \in [-\tau_1, 0]}\xi_p(s).
\]
This further gives
\[
\frac{d}{dt}\theta_p(t) \leq C\theta_p(t) + C\theta_p(t-\tau_2) +C\sup_{s \in [-\tau_1, 0]}\xi_p(s),
\]
and again, we deduce 
\[
\theta_p(t) \leq C\sup_{s \in [-\tau_2, 0]}\theta_p(s)  + C\sup_{s \in [-\tau_1, 0]}\xi_p(s)
\]
for some constant $C>0$ independent of $p$. Hence, we have
\begin{align*}
&d_p(\nu_t^1,\nu_t^2) + \left( \frac1m \sum_{j=1}^m |\bar y_j^1(t) - \bar y_j^2(t)|^p\right)^\frac1p \cr
&\quad \leq C \sup_{s \in [-\tau,0]}d_p(g_s^1,g_s^2) + C \sup_{s \in [-\tau,0]}\left( \frac1m \sum_{j=1}^m |\bar y_j^{1,0}(s) - \bar y_j^{2,0}(s)|^p\right)^\frac1p.
\end{align*}
This completes the proof.
\end{proof}

We conclude this section by stating the stability result for the system \eqref{pde2}. The proof follows from a direct adaptation of the arguments developed in Lemma \ref{4.3}, using the same strategy based on Wasserstein distance estimates and characteristic flows. Since no essential new difficulties arise in this case, we omit the detailed proof.

\begin{Lemma}\label{stab}
Let $T>0$, and let $(\bar\mu_t^1,\bar\nu_t^1)$ and $(\bar\mu_t^2,\bar\nu_t^2)$ be two measure-valued solutions of \eqref{pde2} on the time interval $[0,T]$, constructed according to Theorem \ref{1.7}. Then, there exists a constant $\bar{C}>0$, depending on $T$ but independent of $p \in [1,\infty]$, such that 
\[
d_p(\bar \mu_t^1, \bar\mu_t^2) + d_p(\bar \nu_t^1, \bar\nu_t^2) \leq  \bar{C}\sup_{s \in [-\tau,0]} d_p(f_s^1, f_s^2) + \bar{C}\sup_{s \in [-\tau,0]} d_p(g_s^1, g_s^2)
\]
for all $t \in [0,T)$.
\end{Lemma}

%
%
%
%
%
%
\subsection{Mean-field limit and emergence of consensus}

In this part, we provide the details on the proof of the consensus estimate in Theorem \ref{1.7} establishing the consensus behavior of measure-valued solutions to the mean-field systems \eqref{pde1} and \eqref{pde2}, based on a rigorous passage from the particle model \eqref{eqA}. The key ingredient in the argument is the stability results obtained earlier, which allow us to control the distance between the empirical measure solutions of the particle system and the limiting measure-valued solutions.

 We divide the argument into two cases corresponding to the systems \eqref{pde1} and \eqref{pde2}.
 
 \medskip

\noindent {\it Case (i).}  Let $(\bar y_i^0(s), g_s) \in C([-\tau,0];\R^d) \times C([-\tau,0];\mathcal{P}_\infty(\R^d))$ be given initial data. For each $N \in \mathbb{N}$, we construct a particle approximation of $g_s$ by defining
\[
g_s^{N}:=\frac{1}{N} \sum_{i=1}^N \delta_{x_i^{N,0}(s)}, \quad s \in [-\tau,0],
\]
where $x_i^{N,0} \in C([-\tau,0]; \mathbb{R}^d)$ are chosen such that 
\[
\max_{s \in [-\tau,0]} d_\infty(g_s,g_s^{N}) \rightarrow 0 \quad \mbox{as } N \to \infty.
\]
Likewise, we choose $y_i^{N,0} \in C([-\tau,0]; \mathbb{R}^d)$ satisfying
\[
 \max_{s \in [-\tau,0]} \max_{i=1,\dots, m}|y_i^{N,0}(s) - \bar y_i^0(s)| \to 0  \quad \mbox{as } N \to \infty.
 \]

\begin{Remark}
In principle, one could simply set $y_i^{N,0} = \bar y_i^0$ for all $i$ and $N$. 
However, we keep the above more general approximation procedure, since in the treatment of Case (ii) below we do not rely on such an identification and instead work with general approximating sequences. 
Adopting the same framework here makes the two cases fully parallel and simplifies the presentation.
\end{Remark}

Let $\{y_i^N(t)\}_{i=1}^m$ and $\{x_j^N(t)\}_{j=1}^N$ denote the solution to the particle system \eqref{eqA} corresponding to these initial data. Then, by Theorem \ref{decay_Th} and the definition of the diameter $d(t)$ in Definition \ref{def'}, we have
\begin{equation}\label{lt_dis}
d(t) \leq e^{-\gamma(t-2\tau)}D_0,
\end{equation}
for all $t \in [0,T).$ 

We now define the empirical measure
\[
\nu_t^{N}:= \frac{1}{N} \sum_{i=1}^N \delta_{x_i^N(t)},
\]
which is the measure-valued solution to the system \eqref{pde1} in the sense of Definition \ref{solution}. By Lemma \ref{4.3}, there exists a constant $C>0$, independent of $N$, such that 
\begin{align*}
&d_\infty(\nu_t^N,\nu_t) + \max_{i=1,\dots, m} |y_i^N(t) - \bar y_i(t)|   \leq C \sup_{s \in [-\tau,0]}d_\infty(g_s^N,g_s) + C \sup_{s \in [-\tau,0]}\max_{i=1,\dots, m}|y_i^{N,0}(s) - \bar y_i^{0}(s)|.
\end{align*}
This implies that $d(t) \to d^\nu(t)$ and $D_0 \to D^\nu_0$ as $N \to \infty$, and thus, passing to the limit in \eqref{lt_dis}, we obtain
\[
d^\nu(t) \leq  e^{-\gamma(t-2\tau)}D_0^\nu.
\]

\noindent {\it Case (ii).}  The argument is analogous. For given initial data $(\bar f_s, \bar g_s) \in C([-\tau,0];\mathcal{P}_\infty(\R^d)) \times C([-\tau,0];\mathcal{P}_\infty(\R^d))$, we consider approximations
\[
\bar f_s^m:= \frac1m\sum_{i=1}^m \delta_{\bar y_i^{m,0}(s)} \quad \mbox{and} \quad g^m \in C([-\tau,0];\mathcal{P}_\infty(\R^d))
\]
with $\bar y_i^{m,0}(s) \in C([-\tau,0]; \mathbb{R}^d)$ satisfying
\[
\max_{s \in [-\tau,0]}d_\infty(\bar f_s^m, \bar f_s) + \max_{s \in [-\tau,0]} d_\infty (g_s^m, \bar g_s) \to 0 \quad \mbox{as } m \to \infty.
\]
Let $\{y_i^m\}_{i=1}^m$ and $\nu^m$ denote the solutions to the system \eqref{pde1} corresponding to this initial data. Then, applying the result of Case (i), we obtain
\[
d^{\nu^m}(t) \leq  e^{-\gamma(t-2\tau)}D_0^{\nu^m}.
\]
Next, define the empirical measure
\[
\bar \mu_t^m := \frac{1}{m} \sum_{i=1}^m \delta_{\bar y_i^m(t)},
\]
so that the pair $(\bar \mu_t^m, \nu_t^m)$ solves the system \eqref{pde2}. Then, applying the stability estimate in Lemma \ref{stab}, we get
\[ 
d_\infty(\bar \mu_t^m, \bar\mu_t) + d_\infty(\nu_t^m, \bar\nu_t) \leq  C\sup_{s \in [-\tau,0]} d_\infty(\bar f_s^m, \bar f_s) + C\sup_{s \in [-\tau,0]} d_\infty(g_s^m, g_s).
\] 
As before, taking the limit as $m \to \infty$, we conclude
\[
d^{\bar \mu, \bar \nu}(t) \leq  e^{-\gamma(t-2\tau)}D_0^{\bar \mu, \bar \nu}.
\]
This completes the proof.

%
%
%
%
%
%
%
%
\section*{Acknowledgments}
The work of Y.-P. Choi is supported by NRF grant no. RS-2024-00406821.
C. Cicolani and C. Pignotti are members of Gruppo Nazionale per l'Analisi Matematica, la
Probabilità e le loro Applicazioni (GNAMPA) of the Istituto Nazionale di Alta Matematica (INdAM).
They are partially supported by INdAM -
GNAMPA Projects (CUP E5324001950001). 
C. Pignotti is also partially supported by  PRIN-PNRR 2022 (P20225SP98) {\it Some mathematical approaches to climate
change and its impacts}, and 
 PRIN 2022 (2022238YY5) {\it Optimal control
problems: analysis, approximation, and applications}.

%
%
%
%
%
%

\end{document}